\documentclass[draft,utf8,12pt,english,nothms]{aclart}
\usepackage{lmodern,amssymb,xspace}
 \usepackage{fourier}
\usepackage[draft=false]{hyperref}
\DeclareUnicodeCharacter{00A0}{\relax}
\usepackage{amsmath,amsthm,amssymb}
\usepackage[utf8]{inputenc}
\DeclareUnicodeCharacter{00A0}{\relax}
\usepackage{xspace,url}

\let\mathscr\mathcal

 \let\subsection\Subsection
 \numberwithin{paragraph}{subsection}
 \numberwithin{equation}{section}
 \makeatletter
 \def\paragraph{\@startsection{paragraph}{4}%
   \z@{.4\baselineskip}{-\fontdimen2\font}%
   {\normalfont\scshape}}
 \makeatother

\makeatletter
   \let\c@equation\c@paragraph
   \let\theequation \theparagraph
  \let\cl@equation\cl@paragraph
 \makeatother

\makeatletter
\swapnumbers
\def\th@plain{%
  \let\thm@indent\noindent
  \thm@headfont{\bfseries}%
  \thm@notefont{\upshape}%
  \thm@preskip.5\baselineskip \@plus .5\baselineskip
  \thm@postskip\thm@preskip
  \thm@headpunct{\pointrait}
  \itshape }
\makeatother

\theoremstyle{plain}
\newtheorem{CLprop}[paragraph]{Proposition}
\newtheorem{CLtheo}[paragraph]{Theorem}
\newtheorem{CLcoro}[paragraph]{Corollary}
\newtheorem{CLlemm}[paragraph]{Lemma}

\theoremstyle{definition}
\newtheorem{CLdefi}[paragraph]{Definition}

\theoremstyle{remark}

\begin{document}

\title[Heights and measures on analytic spaces]
 {Heights and measures on analytic spaces. \\
  A survey of recent results, and some remarks}
\author{Antoine Chambert-Loir}
\address{IRMAR, Universit\'e de Rennes~1,
Campus de Beaulieu, 35042 Rennes Cedex.
\url{antoine.chambert-loir@univ-rennes1.fr}}


\subjclass{14G40,14G22}

\maketitle
\tableofcontents

\def\red{{\operatorname{red}}}
\def\curv{{\operatorname{curv}}}
\def\Div{{\operatorname{Div}}}
\def\gm{\mathbf G_{\mathrm m}}
\def\C{{\mathbf C}}
\def\A{{\mathbf A}}
\def\P{{\mathbf P}}
\def\Q{{\mathbf Q}}
\def\N{{\mathbf N}}
\def\R{{\mathbf R}}
\def\Z{{\mathbf Z}}
\def\sp{\operatorname{sp}}
\def\ord{\operatorname{ord}}
\def\dim{\operatorname{dim}}
\let\ra\rightarrow
\let\hra\hookrightarrow
\let\eps\varepsilon
\let\phi\varphi
\let\bar\overline
\let\hat\widehat
\def\resp{\emph{resp.}\xspace}
\def\loccit{\emph{loc.cit.}\xspace}
\def\cf{\emph{cf.}\xspace}
\def\ie{\emph{i.e.}\xspace}
\def\an{{\text{\upshape an}}}
\def\abs#1{\left|{#1}\right|}
\def\norm#1{\left\| {#1} \right\|}

\def\hc{{\widehat c}}
\def\hCH{{\widehat{\operatorname{CH}}}}
\def\hPic{{\widehat{\operatorname{Pic}}}}
\def\ad{{\text{\upshape ad}}}
\def\sm{{\text{\upshape sm}}}
\def\hdiv{\mathop{\widehat{\operatorname{div}}}}
\def\hvol{\mathop{\widehat{\operatorname{vol}}}}
\def\vol{\operatorname{vol}}
\def\Arg{\operatorname{Arg}}

\def\CH{\operatorname{CH}}
\def\Gal{\operatorname{Gal}}
\def\Spf{\operatorname{Spf}}
\def\Spec{\operatorname{Spec}}
\def\Proj{\operatorname{Proj}}
\def\div{\operatorname{div}}
\def\Pic{\operatorname{Pic}}
\def\bPic{\overline{\Pic}}
\def\alg{{\text{\upshape alg}}}
\def\bPicalg{\overline{\Pic}_{\text{\upshape alg}}}
\def\bPicfor{\overline{\Pic}_{\text{\upshape for}}}
\def\bPicsm{\overline{\Pic}_{\text{\upshape sm}}}
\def\bPicad{\overline{\Pic}_{\text{\upshape ad}}}
\def\tube#1{\mathopen]#1\mathclose[}
\def\sozat{\,;\,}
\def\ddc{\mathop{\mathrm d\mathrm d^c}}
\allowdisplaybreaks[1]
\def\FS{{\mathrm {FS}}}
\def\W{{\mathrm {W}}}



The first goal of this paper was to survey my definition
in~\cite{chambert-loir2006}
of measures on non-archimedean analytic spaces in the
sense of Berkovich and to explain its applications
in Arakelov geometry. These measures are analogous
the measures on complex analytic spaces given by
products of first Chern forms of hermitian line bundles.\footnote{%
M.~Kontsevich and Yu.~Tschinkel gave me copies of unpublished notes
from the years 2000--2002 where they develop similar ideas to
construct canonical non-archimedean metrics on Calabi--Yau varieties ;
see also~\cite{kontsevich-soibelman2001,kontsevich-soibelman2006}.}
In both contexts, archimedean and non-archimedean, they are related
with Arakelov geometry and the local height pairings
of cycles. However, while the archimedean measures lie at
the ground of the definition of the archimedean local heights
in Arakelov geometry, the situation is reversed in the ultrametric case:
we begin with the definition of local heights given by arithmetic
intersection theory and define measures in such a way that
the archimedean formulae make sense and are valid.
The construction is outlined
in Section~1, with references concerning
metrized line bundles and the archimedean setting.
More applications to Arakelov geometry
and equidistribution theorems are discussed in Section~3.

The relevance of Berkovich spaces in Diophantine geometry has
now made been clear by many papers; besides~\cite{chambert-loir2006}
and~\cite{chambert-loir-thuillier2009} and the general equidistribution
theorem of Yuan~\cite{yuan2008}, I would like
to mention the works~\cite{gubler2007a,gubler2007b,gubler2008,faber2009}
who discuss the function field case of the equidistribution theorem,
as well as the potential theory
on non-archimedean curves developed simultaneously by
Favre, Jonsson \& Rivera-Letelier~\cite{favre-jonsson2004,favre-rl2007}
and Baker \& Rumely for the projective line~\cite{baker-rumely2010},
and in general by A.~Thuillier's PhD thesis~\cite{thuillier2005}.
The reader will find many important results in the latter work,
which unfortunately is still unpublished at the time of this writing.

Anyway, I found useful to add examples and complements
to the existing (and non-) litterature. This is done in Section~2.
Especially, I discuss in Section~2.2
the relation between the reduction graph
and the skeleton of a Berkovich curve, showing that the two
constructions of measures coincide.
Section~2.3 shows that the measures defined are of a local
nature; more generally, we show that the measures
vanish on any  open subset  where one of the metrized line bundles
involved is trivial. This suggests a general definition  of
\emph{strongly pluriharmonic} functions on Berkovich spaces,
as uniform limits of logarithms of absolute
values of invertible holomorphic functions.
(Strongly pluriharmonic
fonctions should only exhaust pluriharmonic functions
when the residue field is algebraic over a finite field,
but  not in general.)
In Section~2.4, we discuss polarized dynamical systems
and explain the construction of canonical metrics and measures
in that case. We also show that the canonical measure
vanishes on the Berkovich equicontinuity locus. In fact,
what we show is that the canonical metric is ``strongly pluriharmonic''
on that locus. This is the direct generalization of a theorem
of~\cite{rivera-letelier2003} for the projective line
(see also~\cite{baker-rumely2010} for an exposition);
this generalizes also a theorem of~\cite{kawaguchi-silverman2009}
that Green functions are locally constant on the classical
equicontinuity locus. As already were their proofs, mine
is a direct adaptation of the proof
of the complex case~\cite{hubbard-papadopol1994}.
In Section~2.5, following Gubler~\cite{gubler2010},
we finally describe the canonical measures in the case
of abelian varieties.

In Section~3, we discuss applications of the measures in
Diophantine geometry over global fields. Once definitions
are recalled out in Section~3.1, we briefly discuss
in Section~3.2  the relation between Mahler measures
(\ie, integration of Green functions against measures)
and heights.
In Section~3.3, we survey the equidistribution theorems
for Galois orbits of points of ``small height'', following
the variational method of Szpiro--Ullmo--Zhang~\cite{szpiro-u-z97}
 and~\cite{yuan2008}.
In fact, we describe the more general statement
from~\cite{chambert-loir-thuillier2009}.
Finally, Section~3.4 discusses positive lower bounds  for heights
on curves.
This is inspired by recent papers~\cite{baker-demarco2009,petsche-szpiro-tucker2009} but the method goes back to Mimar's unpublished
thesis~\cite{mimar1997}. A recent preprint~\cite{yuan-zhang2009}
of Yuan and Zhang establishes a similar result in any dimension.

\paragraph*{Acknowledgments\pointrait}
This paper grew out of the invitation of Johannes Nicaise and Julien
Sebag
to add a contribution to the proceedings of the
conference ``Motivic integration and its interactions
with model theory and non-archimedean geometry'' (ICMS, Edinburgh 2008);
I thank them heartily for that.

I wrote this paper during a stay at
the Institute for Advanced Study in Princeton whose winterish
but warm atmosphere   was extremly motivating.
I acknowledge support of the Institut
universitaire de France, as well of the National Science
Foundation under agreement No.~DMS-0635607.

During the writing of this paper,
transatlantic e-contacts with Antoine Ducros
have been immensely profitable. Besides him, I also wish to thank
Matthew Baker and Amaury Thuillier for their interest and comments.
I am also grateful to the referee for having pointed out
many misprints as well as some serious inacurracies.

\def\myparagraph#1{\paragraph{#1\pointrait}}
\let\myparagraph\paragraph

\section{Metrized line bundles and measures}

\subsection{Continuous metrics}

\myparagraph{Definition}
Let $X$ be a topological space together
with a sheaf of local rings~$\mathscr O_X$ (``analytic functions'');
let also $\mathscr C_X$ be the sheaf of continuous functions on~$X$.
In analytic geometry, local functions have an absolute value
which is a real valued continuous function,
satisfying the triangle inequality. Let us thus assume
that we have a morphism of sheaves $\mathscr O_X\ra\mathscr C_X$,
written $f\mapsto \abs f$,
such that $\abs{fg}=\abs f\abs g$,
$\abs 1=1$, and $\abs{f+g}\leq\abs f+\abs g$.

A line bundle on~$(X,\mathscr O_X)$ is a sheaf~$L$ of~$\mathscr O_X$-modules
which is locally isomorphic to~$\mathscr O_X$. In other words,
$X$ is covered by open sets~$U$ such that $\mathscr O_U\simeq L|U$;
such an isomorphism is equivalent to a non-vanishing
section~$\eps_U\in\Gamma(U,L)$, also called a local frame of~$L$.

If $s$ is a section of a line bundle~$L$ on an open set~$U$,
the value of~$s$ at a point~$x\in U$ is only well-defined
as an element of the stalk~$L(x)$, which is a $\kappa(x)$-vector
space of dimension~$1$. (Here, $\kappa(x)$ is the residue
field of~$\mathscr O_X$ at~$x$.)
Prescribing a  metric on~$L$ amounts to assigning, in a coherent way,
the norms of these values.
Formally,
a \emph{metric} on~$L$ is the datum, for any open set~$ U\subset X$
and any section $s\in\Gamma(U,L)$, of a continuous
function $\norm s_U\colon U\ra\R_+$, satisfying the following properties:
\begin{enumerate}
\item for any open set $V\subset U$, $\norm s_V$ is the restriction
to~$V$ of the function~$\norm s_U$;
\item for any function $f\in\mathscr O_X(U)$,
$\norm {fs}=\abs f \norm s$;
\item if $s$ is a local frame on~$U$, then $\norm s$ doesn't vanish
at any point of~$U$.
\end{enumerate}
One usually writes $\bar L$ for the pair $(L,\norm \cdot)$ of
a line bundle~$L$ and a metric on it.

Observe that the trivial line bundle~$\mathscr O_X$
has a natural ``trivial'' metric, for which $\norm 1=1$.
In fact, a metric on the trivial line bundle~$\mathscr O_X$
is equivalent to the datum of a continuous function~$h$ on~$X$,
such that $\norm 1=e^{-h}$.

\myparagraph{The Abelian group of metrized line bundles}
Isomorphism of metrized line bundles
are isomorphisms of line bundles which respect the metrics;
they are called \emph{isometries}.
Constructions  from tensor algebra extend naturally to the framework
of metrized line bundles, compatibly with isometries.
The tensor product of two
metrized line bundles $\bar L$ and~$\bar M$ has a natural metrization
such that $\norm {s\otimes t}=\norm s \norm t$, if $s$ and~$t$
are local sections of~$L$ and~$M$ respectively.
Similarly, the dual of a metrized line bundle has a metrization,
and the obvious isomorphism $L\otimes L^\vee\simeq \mathscr O_X$
is an isometry.
Consequently, isomorphism  classes
of metrized line bundles on~$X$ form an Abelian group~$\bPic(X)$.
This group fits in an exact sequence
\[ 0 \ra \mathscr C(X) \ra \bPic(X) \ra \Pic(X) , \]
where the first map associates to a real continuous function~$h$
on~$X$ the trivial line bundle endowed with the metric
such that $\norm 1=e^{-h}$, and the second associates
to a metrized line bundle the underlying line bundle.
It is surjective
when any line bundle has a metric (this certainly holds
if $X$ has partitions of unity).

Similarly, we can consider pull-backs of metrized line bundle.
Let $\phi\colon Y\ra X$ be a morphism of locally ringed spaces
such that $\abs {\phi^*f}=\abs f\circ\phi$ for any $f\in\mathscr O_X$.
Let $\bar L$ be a metrized line bundle on~$X$.
Then, there is a canonical metric on~$\phi^*L$ such that
$\norm{\phi^*s}=\norm s\circ\phi$ for any section~$s\in\Gamma(U,L)$.
This induces a morphism of Abelian groups $\phi^*\colon\bPic(X)\ra\bPic(Y)$.

\subsection{The case of complex analytic spaces}
\myparagraph{Smooth metrics}
In complex analytic geometry, metrics are a very
well established tool.
Let us first consider the case of the projective
space $\mathrm X=\P^n(\C)$; a point $x\in X$ is a $(n+1)$-tuple
of homogeneous coordinates $[x_0:\dots:x_n]$, not all zero,
and up to a scalar. Let $\pi\colon \C^{n+1}_*\ra X$ be the canonical
projection map, where the index~$*$ means that we remove
the origin~$(0,\dots,0)$.
The fibers of~$\pi$ have a natural action of~$\C^*$ and
the line bundle~$\mathscr O(1)$
has for sections~$s$ over an open set~$\mathrm U\subset\P^n(\C)$
the analytic functions~$F_s$
on the open set~$\pi^{-1}(\mathrm U)\subset\C^{n+1}_*$
which are homogeneous of degree~$1$.
The \emph{Fubini-Study metric} of~$\mathscr O(1)$ assigns
to the section~$s$ the norm $\norm s_\FS$ defined by
\[  \norm s_{\FS}([x_0:\dots:x_n]) = \frac{\abs{F_s(x_0,\dots,x_n)}}{\left(\abs{x_0}^2+\dots+\abs{x_n}^2\right)^{1/2}}. \]
It is more than continuous; indeed, if $s$ is a local frame on an open set~$\mathrm U$,
then $\norm s$ is a $\mathscr C^\infty$-function on~$\mathrm U$;
such metrics are called \emph{smooth}.

\myparagraph{Curvature}
Line bundles with smooth metrics
on smooth complex analytic spaces allow to perform
differential calculus. Namely, the \emph{curvature} of a
smooth metrized line bundle~$\bar L$
is a differential form $c_1(\bar L)$ of type~$(1,1)$ on~$\mathrm X$.
Its definition involves the differential operator
\[ \ddc=\frac i\pi \partial\bar\partial .\]
When an open set $\mathrm  U\subset \mathrm  X$
admits local coordinates $(z_1,\dots,z_n)$,
and $s\in\Gamma(\mathrm U,L)$ is a local frame, then
\[ c_1(\bar L)|_{\mathrm U} = \ddc \log\norm{s}^{-1}
 =\frac i\pi \sum_{1\leq j,k\leq n} \frac{\partial^2}{\partial z_j\partial \bar z_k}
\log\norm{s}^{-1} \mathrm dz_j\wedge \mathrm d\bar z_k. \]
The Cauchy-Riemann equations ($\partial f/\partial \bar z=0$ for
any holomorphic function~$f$ of the variable~$z$) imply
that this formula does not depend on the choice
of a local frame~$s$. Consequently, these differential forms
defined locally glue to
a well-defined global differential form on~$\mathrm X$.

Taking the curvature form of a metrized line bundle is a linear
operation: $c_1(\bar L\otimes\bar M)=c_1(\bar L)+c_1(\bar M)$.
It also commutes with pull-back: if $f\colon Y\ra X$ is a morphism,
then $f^*c_1(\bar L)=c_1(f^*\bar L)$.

In the case of the Fubini-Study metric over the projective space~$\P^n(\C)$,
the curvature is computed as follows.
The open subset~$\mathrm U_0$ where the homogeneous coordinate~$x_0$ is non-zero
has local coordinates $z_1=x_1/x_0$,
\dots, $z_n=x_n/x_0$;  the homogeneous polynomial~$X_0$
defines a non-vanishing section~$s_0$ of~$\mathscr O(1)$ on~$\mathrm U_0$
and \[ \log\norm{s_0}^{-1}_\FS = \frac12 \log\left(1+\sum_{j=1}^n\abs{z_j}^2\right). \]
Consequently, over~$\mathrm U_0$,
\begin{align*}
 c_1(\overline{\mathscr O(1)}_\FS) & =
\frac i\pi \partial\bar\partial \log\norm{s_0}^{-1}_\FS  \\
&  =\frac i{2\pi} \partial\left( \sum_{k=1}^n \frac{z_k}{1+\norm z^2} \mathrm d\bar z_k\right) \\
& =\frac i{2\pi} \sum_{j=1}^n \frac1{1+\norm z^2} \mathrm dz_j\wedge\mathrm d\bar z_j
- \frac i{2\pi}\sum_{j,k=1}^n \frac{z_k\bar z_j}{(1+\norm z^2)^2}
\mathrm dz_j\wedge\mathrm d\bar z_k.
\end{align*}
In this calculation, we have abbreviated $\norm z^2=\sum_{j=1}^n\abs{z_j}^2$.

\myparagraph{Products, measures}
Taking the product  of $n$ factors equal to this differential form,
we get a differential form of type~$(n,n)$ on the $n$-dimensional
complex space~$\mathrm X$. Such a form can be integrated on~$\mathrm X$
and the Wirtinger formula asserts that
\[ \int_{\mathrm X} c_1(\bar L)^n = \deg (L) \]
is the \emph{degree} of~$L$ as computed by intersection theory.
As an example, if $\mathrm X=\P^1(\C)$, we have seen that
\[ c_1(\overline{\mathscr O(1)}_\FS)= \frac{i}{2\pi(1+\abs z^2)^2}  \mathrm dz\wedge d\bar z, \]
where $z=x_1/x_0$ is the affine coordinate of~$X\setminus\{\infty\}$.
Passing in polar coordinates $z=r e^{i\theta}$, we get
\[ c_1(\overline{\mathscr O(1)}_\FS)
= \frac{1}{2\pi(1+r^2)^2}
 2r\mathrm dr\wedge\mathrm d\theta \]
whose integral over~$\C$ equals
\[ \int_{\P^1(\C)} c_1(\overline{\mathscr O(1)}_\FS)
= \int_0^\infty \frac{1}{2\pi(1+r^2)^2}  2r\mathrm dr
 \int_0^{2\pi} \mathrm d\theta
= \int_0^\infty \frac{1}{(1+u)^2}\mathrm du
= 1 . \]

\myparagraph{The Poincaré--Lelong equation}
An important formula is the Poincaré--Lelong equation.
For any line bundle~$L$ with a smooth metric,
and any section~$s\in\Gamma(\mathrm X,L)$ which does not vanish identically
on any connected component of~$\mathrm X$, it asserts the following
equality of \emph{currents}\footnote{The space of currents is the dual
to the space of differential forms, with the associated grading;
in the orientable case,
currents can also be seen as differential forms with distribution coefficients.}:
\[ \ddc \log\norm s^{-1} + \delta_{\div(s)} = c_1(\bar L), \]
where $\ddc\log\norm s^{-1}$ is the image of~$\log\norm{s}^{-1}$
under the differential operator~$\ddc$, taken in the sense of distributions,
and $\delta_{\div(s)}$ is the current of integration on the cycle~$\div(s)$
 of codimension~$1$.

\myparagraph{Archimedean height pairing}\label{sec.local.height.arch}
Metrized line bundles and their associated curvature forms are
a basic tool in Arakelov geometry, invented by Arakelov in~\cite{arakelov74}
and developped by Faltings~\cite{faltings1984}, Deligne~\cite{deligne1987}
for curves, and by Gillet-Soulé~\cite{gillet-s90} in any dimension.
For our concerns, they allow for a definition of height functions
for algebraic cycles on algebraic varieties defined over number fields.
As explained by Gubler~\cite{gubler1997,gubler1998},
they also permit to develop a theory of archimedean local heights.

For simplicity, let us assume that $\mathrm X$ is proper, smooth, and that
all of its connected components have dimension~$n$.

Let $\bar L_0,\dots,\bar L_n$ be metrized line bundles
with smooth metrics. For $j\in\{0,\dots,n\}$, let $s_j$
be a regular meromorphic section of~$L_j$ and let $\div(s_j)$
be its divisor. The given metric of~$L_j$ furnishes
moreover a function $\log\norm{s_j}^{-1}$ on~$X$
and a $(1,1)$-form $c_1(\bar L_j)$, related by the Poincaré--Lelong
equation $\ddc\log\norm{s_j}^{-1}+\delta_{\div(s_j)}=c_1(\bar L_j)$.
In the terminology of Arakelov geometry,
$\log\norm{s_j}^{-1}$ is a Green current (here, function)
for the cycle $\div(s_j)$; we shall write $\hdiv(s_j)$
for the pair $(\div(s_j),\log\norm{s_j}^{-1})$.

Let $\mathrm Z\subset \mathrm X$ be a $k$-dimensional subvariety
such that the divisors $\div(s_j)$,
for $0\leq j\leq k$, have no common point on~$\mathrm Z$.
Then, one defines inductively the local height pairing
by the formula:
\begin{multline}\label{eq.inductive-C}
 (\hdiv(s_0)\dots \hdiv(s_k)|\mathrm Z) = (\hdiv(s_0)\dots\hdiv(s_{k-1})|\div(s_k|_{\mathrm Z})) \\
+ \int_{\mathrm X} \log\norm{s_k}^{-1} c_1(\bar L_0)\dots c_1(\bar L_{k-1})\delta_{\mathrm Z}. \end{multline}
The second hand of this formula requires two comments.
1) The divisor $\div(s_k|_{\mathrm Z})$ is a formal linear combination
of $(k-1)$-dimensional subvarieties of~$\mathrm X$,
and its local height pairing is computed by linearity from the
local height pairings of its components.
2) The integral of the right hand side involves a function with
singularities ($\log\norm {s_k}^{-1}$) to be integrated against
a distribution: in this case, this means restricting
the differential form $c_1(\bar L_0)\dots c_1(\bar L_{k-1})$ to the smooth part of~$\mathrm Z$, multiplying by $\log\norm{s_k}^{-1}$, and integrating
the result. The basic theory of closed positive currents proves
that the resulting integral converges absolutely; as in~\cite{gillet-s90},
one can also resort to Hironaka's resolution of singularities.

It is then a non-trivial result that the local height
pairing is symmetric in the involved $\hdiv$isors;
it is also multilinear. See~\cite{gubler2003} for more details,
as well as~\cite{gillet-s90} for the global case.

\myparagraph{Positivity}
Consideration of the curvature allows to define positivity notions
for metrized line bundles.
Namely, one says that a smooth metrized line bundle~$\bar L$ is \emph{positive}
(resp. \emph{semi-positive})
if its curvature  form is a positive (resp. a non-negative) $(1,1)$-form.
This means that for any point $x\in \mathrm X$, the
hermitian form $c_1(\bar L)_x$ on the complex tangent space~$\mathrm T_x\mathrm X$
is positive definite (resp. non-negative).
As a crucial example, the line bundle~$\mathscr O(1)$ with
its Fubini-Study metric is positive. The pull-back
of a positive metrized line bundle by an immersion is
positive. In particular, ample line bundles can be endowed
with a positive smooth metric; Kodaira's embedding theorem
asserts the converse: if a line bundle possesses a positive
smooth metric, then it is ample.

The pull-back of a semi-positive metrized line bundle
by any morphism is still semi-positive.
If $\bar L$ is semi-positive, then the measure $c_1(\bar L)^n$
is a positive measure.

\myparagraph{Semi-positive continuous metrics}
More generally, both the curvature and the Poincaré--Lelong equation
make sense for metrized line bundles with arbitrary (continuous)
metrics, except that $c_1(\bar L)$ has to be considered as a current.
The notion of semi-positivity can even be extended to this more
general case, because it can be tested by duality: a current
is positive if its evaluation on any nonnegative differential form
is nonnegative. Alternatively,
semi-positive (continuous) metrized line bundles are characterized by
the fact that for any local frame~$s$ of~$\bar L$ over an open set~$\mathrm U$,
the continuous
function $\log\norm{s}^{-1}$ is \emph{plurisubharmonic} on~$\mathrm U$.
In turn, this means that for any morphism $\phi\colon \bar D \ra \mathrm U$,
where $\bar D=\bar D(0,1)$ is the closed unit disk in~$\C$,
\[ \log\norm{s}^{-1} (\phi(0)) \leq \frac1{2\pi}\int_0^{2\pi} \log\norm{s}^{-1}(\phi(e^{i\theta}))\mathrm d\theta. \]

Assume that $\bar L$ is semi-positive.
Although products of currents are not defined in general
(not more than products of distributions), the theory of
Bedford--Taylor~\cite{bedford-t82,bedford93}
and Demailly~\cite{demailly1985,demailly93} defines a current
$c_1(\bar L)^n$ which then is a positive measure on~$\mathrm X$.
There are two ways to define this current. The first one
works locally and proceeds by induction: if $u=\log\norm s^{-1}$,
for a local non-vanishing section~$s$ of~$L$, one defines
a sequence~$(T_k)$ of closed positive currents by
the formulae $T_0=1$, $T_1=\ddc u$,\dots,
$T_{k+1}=\ddc(uT_k)$ and $c_1(\bar L)^n=\ddc(u)^n$ is defined to be~$T_n$.
What makes this construction work is the fact that at each step,
$uT_k$ is a well-defined current (product of a continuous function and of a positive current), and one has to prove
that $T_{k+1}$ is again a closed positive current.
The other way, which shall be the one akin to a generalization
in the ultrametric framework, consists in observing that
if $L$ is a line bundle with a continuous semi-positive metric~$\norm\cdot$,
then there exists a sequence of smooth semi-positive metrics $\norm\cdot_k$
on the line bundle~$L$ which converges uniformly to the initial metric:
for any local section~$s$, $\norm s_k$ converges uniformly to~$\norm s$
on compact sets.
The curvature current $c_1(\bar L)$ is then the limit of
the positive currents $c_1(\bar L_k)$, and the measure $c_1(\bar L)^n$
is the limit of the measures $c_1(\bar L_k)^n$.
(We refer to~\cite{maillot2000} for the global statement;
to construct the currents, one can in fact work locally in
which case a simple convolution argument establishes the claim.)

An important example of semi-positive metric which is continuous,
but not smoth, is furnished by the Weil metric on the
line bundle~$\mathscr O(1)$ on~$\P^n(\C)$. This metric is defined
as follows: if $\mathrm U\subset\P^n(\C)$ is an open set, and $s$
is a section of~$\mathscr O(1)$ on~$U$ corresponding
to an analytic function~$F_s$ on~$\pi^{-1}(\mathrm U)\subset\C^{n+1}_*$
which is homogeneous of degree~$1$, then for any $(x_0,\dots,x_n)\in\pi^{-1}(\mathrm U)$, one has
\[ \norm s_\W = \frac{\abs{F_s(x_0,\dots,x_n)}}{\max(\abs{x_0},\dots,\abs{x_n})}.\]
The associated measure $c_1(\overline{\mathscr O(1)}_{\mathrm W})^n$
on~$\P^n(\C)$ is as follows,
cf.~\cite{zhang95,maillot2000}:
the subset
of all points $[x_0:\dots:x_n]\in\P^n(\C)$ such that $\abs{x_j}=\abs{x_k}$
for all $j,k$ is naturally identified with the polycircle $\mathbf S_1^n$
(map $[x_0:\dots:x_n]$ to $(x_1/x_0,\dots,x_n/x_0)$); take the normalized
Haar measure of this compact group and push it onto~$\P^n(\C)$.

\myparagraph{Admissible metrics}
Let us say that a continuous metrized line bundle
is \emph{admissible}  if it can be written as $\bar L\otimes\bar M^\vee$,
where $\bar L$ and~$\bar M$  are  metrized line bundles
whose metrics are continuous and semi-positive.
Admissible metrized line bundles form a subgroup $\bPicad(\mathrm X)$ of~$\bPic(\mathrm X)$
which maps surjectively onto~$\Pic(\mathrm X)$ if $\mathrm X$ is projective.

The curvature current~$c_1(\bar L)$
of an admissible metrized line bundle~$\bar L$
is a differential form of type~$(1,1)$ whose coefficients are
signed measures. Its $n$th product $c_1(\bar L)^n$
is well-defined as a signed measure on~$\mathrm X$.

\myparagraph{Local height pairing (admissible case)}\label{sec.local.height.arch.ad}
The good analytic properties of semi-positive metrics
allow to extend the definition of the local height pairing
to the case of admissible line bundles.
Indeed, when one approximates uniformly a semi-positive line bundle
by a sequence of smooth semi-positive line bundles, one can prove
that the corresponding sequence of local height pairings
converges, the limit being independent on the chosen approximation.

The proof is inspired by Zhang's proof of the global
case in~\cite{zhang95b} and goes by induction.
Let us consider, for each~$j$, two smooth semi-positive metrics
on the line bundle~$L_j$ and assume that they differ by a factor~$e^{-h_j}$.
Then, the corresponding
local height pairings differ from an expression of the form
\[  \sum_{j=0}^k \int_{\mathrm Z} h_j c_1(\bar L_0)\dots \widehat{c_1(\bar L_j)} \dots c_1(\bar L_{k}), \]
where the written curvature forms are associated to the first metric
for indices~$<j$, and to the second for indices~$>j$.
This differential forms are positive by assumption, so that
the integral
is bounded  in absolute value by
\begin{multline*} \sum_{j=0}^k \norm {h_j}_\infty \int_{\mathrm Z} c_1(\bar L_0)\dots
\widehat{c_1(\bar L_j)} \dots  c_1(\bar L_{k}) \\
= \sum_{j=0}^K \norm {h_j}_\infty (c_1(L_0)\dots \widehat {c_1(L_j)}
\dots c_1(L_{k})|{\mathrm Z}),\end{multline*}
where the last expression is essentially a degree. (In these
formulae, the factor with a hat is removed.)
This inequality means that on the restriction to
the space of smooth semi-positive metrics, with the topology
of uniform convergence, the local height
pairing is uniformly continuous. Therefore, it first extends by continuity
on the space of continuous semi-positive metrics,
and then by multilinearity to the space of admissible metrics.

\subsection{The case of non-archimedean analytic spaces}\label{sec:meas-na}
Let $K$ be a complete ultrametric field. We are principally
interested in finite extensions of~$\Q_p$,
but the case of local fields of positive characteristic
(finite extensions of~$k((T))$, for a finite field~$k$)
have proved being equally useful, as are non-local fields
like the field~$\C((T))$ of Laurent power series with \emph{complex}
coefficients.
For simplicity, we will assume that $K$ is the
field of fractions of a complete discrete valuation ring~$K^\circ$,
let $\pi$ be a generator of the maximal ideal of~$K^\circ$
and let $\tilde K=K^\circ/(\pi)$ be the residue field.

\myparagraph{Continuous metrics}
Let $\mathrm X$
be a $K$-analytic space in the sense of Berkovich~\cite{berkovich1990}.
For simplicity, we will assume that $\mathrm X$ is
the analytic space associated to a \emph{proper} scheme over~$K$.
In that context, the general definition of continuous metrized
line bundles given above makes sense.

Let us detail the example of the line bundle~$\mathscr O(1)$
on the projective space~$\mathrm P^n_K$.
A point~$x\in\mathrm P^n_K$ possesses a complete residue field~$\mathscr H(x)$
which is a complete extension of~$K$ and homogeneous
coordinates $[x_0:\dots:x_n]$ in the field~$\mathscr H(x)$.
As in complex geometry, the projective space $\mathrm P^n_K$ is obtained
by glueing $n+1$ copies $\mathrm U_0,\dots,\mathrm U_n$ of the affine space $\mathrm A^n_K$,
where $\mathrm U_i$ corresponds to those points~$x$ such that $x_i\neq 0$.
Recall also that $\mathrm A^n_K$ is the space of multiplicative
semi-norms on the $K$-algebra $K[T_1,\dots,T_n]$ which induce
the given absolute value  on~$K$, together with the coarsest
topology such that for any semi-norm $x\in\mathrm A^n_K$,
the map $K[T_1,\dots,T_n]\ra\R$ defined by $f\mapsto x(f)$
is continuous. The kernel of a semi-norm~$x$ is a prime ideal~$\mathfrak p_x$
of~$K[T_1,\dots,T_n]$ and $x$ induces a norm on
the quotient ring $K[T_1,\dots,T_n]/\mathfrak p_x$,
hence on its field of fractions~$K(x)$. The completion
of $K(x)$ with respect to this norm is denoted~$\mathscr H(x)$
and is called the \emph{complete residue field} of~$x$.
The images in~$\mathscr H(x)$ of the intederminates~$T_i$ are denoted~$T_i(x)$,
more generally, the image in~$\mathscr H(x)$
of any polynomial~$f\in K[T_1,\dots,T_n]$ is denoted $f(x)$;
one has $x(f)=\abs{f(x)}$.

Let $f$ be a rational function on~$\mathrm P^n_K$, that is an element
of~$K(T_1,\dots,T_n)$. It defines an actual function
on the open set~$\mathrm U$ of~$\P^n_K$ where its denominator does not vanish;
its value at a point $x\in U$ is an element of~$\mathscr H(x)$.
More generally, Berkovich defines an analytic function on an open set~$\mathrm U$
of~$\mathrm P^n_K$ as a function $f$ on~$U$ such that $f(x)\in\mathscr H(x)$
for any $x\in \mathrm U$, and such that any point $x\in\mathrm U$
possesses a neighbourhood~$\mathrm V\subset\mathrm U$
such that $f|_{\mathrm V}$ is a uniform limit of rational functions
without poles on~$\mathrm V$.

The line bundle~$\mathscr O(1)$ can also be defined in a similar
way to the classical case; by a similar GAGA theorem, its global sections
are exactly the same as in algebraic geometry and are
described by homogeneous polynomials of degree~$1$ with coefficients
in~$K$. If $P$ is such a polynomial and~$s_P$ the corresponding
section, then
\[ \norm {s_P} (x) = \frac{ \abs{P(x_0,\dots,x_n)} } {\max(\abs{x_0},\dots,\abs{x_n})} \]
where $[x_0:\dots:x_n]$ is a system of homogeneous coordinates
in~$\mathscr H(x)$ for the point~$x$. The function $\norm {s_P}$
is continuous on~$\mathrm P^n_K$,
by the very definition of the topology on~$\mathrm P^n_K$.
Using the fact that $\mathscr O(1)$ is generated by its global sections,
one deduces the existence of a continuous metric on~$\mathscr O(1)$
satisfying the previous formula.

\myparagraph{Smooth metrics}
Following~\cite{zhang95b},
we now want to explain the analogues of smooth, and, later,
of semi-positive metrics.

Smooth metrics come from algebraic geometry over~$K^\circ$,
and, more generally, over the ring of integers of finite extensions of~$K$.
Let namely $\mathfrak X$ be a formal proper $K^0$-scheme
whose generic fibre in the sense of analytic geometry is~$\mathrm X$.%
\footnote{The reader might want to assume
that $\mathrm X$ is the analytic space associated to a
projective $K$-scheme~$X$ and that $\mathfrak X$ is a projective
$K^0$-scheme whose generic fibre equals~$X$. This doesn't make
too much a difference for our concerns.}
Let also $\mathfrak L$ be a line bundle on~$\mathfrak X$ which is model of
some power~$L^e$, where $e\geq 1$.
From this datum $(\mathfrak X,\mathfrak L,e)$, we can define a metric
on~$L$ as follows. Let $\mathfrak U$ be a formal open subset
of~$\mathfrak X$ over which~$\mathfrak L$ admits a local frame~$\eps_{\mathfrak U}$; over its generic fibre~$\mathrm U=\mathfrak U_K$, for any section~$s$
of~$L$, one can  write canonically $s^e=f\eps_{\mathfrak U}$, where
$f\in\mathscr O_{\mathrm X}(\mathrm U)$.
We decree that $\norm s=\abs f^{1/e}$.
In other words, the norm of a local frame on the formal model is assigned to be identically one. This makes sense because if $\eta_{\mathfrak U}$ is another
local frame of~$\mathfrak L$ on~$\mathfrak U$, there exists
an invertible formal function $f\in\mathscr O_{\mathfrak X}(\mathfrak U)^*$
such that $\eta_{\mathfrak U}=f\eps_{\mathfrak U}$
and the absolute value~$\abs f$ of the associated analytic function
on~$\mathrm U$ is identically equal to~$1$.
Considering a finite cover of~$\mathfrak X$ by formal open subsets,
their generic fibers form a finite cover of~$\mathrm X$
by \emph{closed} subsets and this is enough to glue
the local definitions to a continuous metric on~$L$.

Metrics on~$L$ given by this construction, for some
model~$(\mathfrak X,\mathfrak L,e)$ of some power~$L^e$ of~$L$
will be said to be \emph{smooth}.

\myparagraph{Green functions; smooth functions}
Let $\bar L$ be a metrized line bundle and let $s$
be a regular meromorphic section of~$L$. Its divisor~$\div(s)$
is a Cartier divisor in~$\mathrm X$.
The function $\log\norm{s}^{-1}$ is defined on the open
set $\mathrm X\setminus\abs{\div(s)}$; by analogy to
the complex case, we call it a \emph{Green function} for the
divisor~$\div(s)$.
When the metric on~$\bar L$ is smooth, the Green function
is said to be smooth.  The same remark applies for the other qualificatives
semi-positive, or admissible, that wil be introduced later.

Let us take for~$L$ the trivial line bundle, with
its canonical trivialization~$s=1$, and let us
endow it with a smooth metric. By definition, we
call $\log\norm s^{-1}$ a smooth function.
More generally, we define the space~$\mathscr C^\infty(\mathrm X)$
of (real valued) smooth functions to
be the real vector space spanned by these elementary
smooth functions. Observe that this definition reverses what
happens in complex geometry where smooth metrics on
the trivial line bundle are \emph{defined} from the knowledge
of smooth functions.

\myparagraph{Example: projective space}
Let us consider the
smooth metric on~$\mathscr O(1)$ associated to
the model~$(\mathbf P^n_{K^\circ},\mathscr O(1),1)$  of~$(\mathrm P^n_K,\mathscr O(1))$.
Let $\mathfrak U_i$ be the formal open subset of~$\mathbf P^n_{K^\circ}$
defined by the non-vanishing of the homogeneous coordinate~$x_i$.
Over, $\mathfrak U_i$, $\mathscr O(1)$ has a global non-vanishing
section, namely the one associated to the homogeneous polynomial~$X_i$.
The generic fiber~$U_i$ of~$\mathfrak U_i$ in the sense of algebraic geometry
is an affine space, with coordinates $z_j=x_j/x_i$, for $0\leq j\leq n$,
and $j\neq i$. However, its generic fiber~$\mathrm U_i$ in
the sense of rigid geometry is the $n$-dimensional polydisk in this affine
space defined by the inequalities $\abs{z_j}\leq 1$.
We thus observe that for any $x\in(\mathfrak U_i)_K$,
\begin{align*} \norm{X_i}(x) & =1 \\
& = \frac 1{\max(\abs{z_0},\dots,\abs{z_{i-1}},1,\abs{z_{i+1}},\dots,\abs{z_n})} \\
& = \frac{\abs{x_i}}{\max(\abs{x_0},\dots,\abs{x_i})} \\
& =\norm{X_i}_\W(x). \end{align*}
In other words, the Weil metric on~$\mathscr O(1)$ is a smooth metric.

\myparagraph{The Abelian group of smooth line bundles}
Let us show that \emph{any line bundle has a smooth metric.}
There is a general theory, due to Raynaud, that shows how to define
formal models from rigid analytic objects. In the present case,
$X$ being projective,
we may assume that $L$ is ample and consider a
closed embedding of~$\mathrm X$ in a projective space~$\mathrm P^n_K$
given by some power~$L^e$. Let $\mathfrak X$ be the Zariski closure
of~$\mathrm X$ in~$\mathrm P^n_{K^\circ}$; in concrete terms,
 if $I\subset K[X_0,\dots,X_n]$ is the homogeneous
ideal of~$i(X)$, $I\cap K^\circ[X_0,\dots,X_n]$ is the homogeneous
ideal of~$\mathfrak X$.
Let then $\mathfrak L$ be the restriction
to~$\mathfrak X$ of the line bundle~$\mathscr O(1)$.
The triple $(\mathfrak X,\mathfrak L,e)$ is a model of~$L$
and induces a smooth metric on~$L$.

Different models can give rise to the same metric.
If $\phi\colon\mathfrak X'\ra\mathfrak X$ is a morphism of models,
and $\mathfrak L'=\phi^*\mathfrak L$, then $(\mathfrak X',\mathfrak L',e)$
defines the same smooth metric on~$L$.
Moreover, if
two models~$(\mathfrak X_i,\mathfrak L_i,e_i)$, for $i\in\{1,2\}$,
define the same metric, there exists a third model~$(\mathfrak X,\mathfrak L)$,
with two morphisms $\phi_i\colon\mathfrak X\ra\mathfrak X_i$
such that the pull-backs $\phi_i^*\mathfrak L_i^{e_1e_2/e_i}$
coincide with~$\mathfrak L$.
More precisely,
if two models~$\mathfrak L$ and~$\mathfrak L'$ of some power~$L^e$
on a \emph{normal} model~$\mathfrak X$ define the same metric,
then they are isomorphic. (See, \emph{e.g.}, Lemma~2.2
of~\cite{chambert-loir-thuillier2009};
this may be false for non-normal
models; it suffices that $\mathfrak X$ be integrally closed in its
generic fiber.)

As a consequence, the set $\bPicsm(\mathrm X)$ of smooth metrized line bundles
is a subgroup of the group~$\bPic(\mathrm X)$.
The group $\bPicsm(\mathrm X)$
fits within an exact sequence
\[ 0 \ra \mathscr C^\infty(\mathrm X) \ra \bPicsm(\mathrm X)\ra\Pic(\mathrm X)\ra 0, \]
the last map is surjective because every line bundle admits a model.
If $f\colon \mathrm Y\ra \mathrm X$
is a morphism, then $f^*(\bPicsm(\mathrm X))\subset\bPicsm(\mathrm Y)$.

\myparagraph{Semi-positive metrics}
A smooth metric
is said to be \emph{ample}
if it is defined by a model~$(\mathfrak X,\mathfrak L,e)$
such that the restriction~$\mathfrak L_{\tilde K}$
of~$\mathfrak L$ to the closed fiber~$\mathfrak X_{\tilde K}$ is ample.
The Weil metric on the line bundle~$\mathscr O(1)$ on
the projective space is ample.
The proof given above of the existence of smooth metrics shows,
more precisely, that ample line bundles  admit ample metrics,
and that \emph{the pull-back of a smooth ample metric
by an immersion is a smooth ample metric.}

A smooth metric is said to be
\emph{semi-positive} if it can be defined on a model~$(\mathfrak X,\mathfrak L,e)$ such that the restriction~$\mathfrak L_{\tilde K}$
of~$\mathfrak L$ to the closed fiber~$\mathfrak X_{\tilde K}$ is
numerically effective: 
for any projective curve~$C\subset\mathfrak X_{\tilde K}$,
the degree of the restriction to~$C$ of~$\mathfrak L_{\tilde K}$ 
is non-negative.
Ample metrics are semi-positive.

The pull-back of a smooth semi-positive metric by any morphism
is semi-positive.

\myparagraph{Continuous semi-positive metrics}
Let us say that a continuous metric on a line bundle~$L$
is semi-positive if it is the uniform limit of a sequence
of smooth semi-positive metrics on the same line bundle~$L$.
As in the complex case, we then say that a metrized
line bundle is \emph{admissible} if it can be written
as $\bar L\otimes\bar M^\vee$, for two line bundles~$L$
and~$M$ with continuous semi-positive metrics.

Let $L$ be a metrized line bundle, and let $\norm\cdot_1$
and $\norm\cdot_2$ be two continuous metrics on~$L$.
It follows from the definition
that the metrics $\norm\cdot_{\min}=\min(\norm\cdot_1,\norm\cdot_2) $
and $\norm\cdot_{\max}=\max(\norm\cdot_1,\norm\cdot_2)$
are continuous metrics.

Moreover, \emph{these metrics $\norm\cdot_{\min}$ and $\norm\cdot_{\max}$
are smooth if the initial metrics are smooth.}
Indeed, there exists a model~$\mathfrak X$,
as well as two line bundles~$\mathfrak L_1$
and~$\mathfrak L_2$ extending the same power~$L^e$
of~$L$ and defining the metrics~$\norm\cdot_1$ and~$\norm\cdot_2$
respectively.
We may assume that $\mathfrak L_1$ and $\mathfrak L_2$ have
regular global sections~$s_1$ and~$s_2$ on~$\mathfrak X$
which coincide on~$X$,
with divisors~$\mathfrak D_1$ and~$\mathfrak D_2$ respectively.
(The general case follows, by
twisting~$\mathfrak L_1$ and~$\mathfrak L_2$
by a sufficiently ample line bundle on~$\mathfrak X$.)
The blow-up $\pi\colon\mathfrak X'\ra\mathfrak X$
of the ideal~$\mathfrak I_{\mathfrak D_1}+\mathfrak I_{\mathfrak D_2}$
carries an invertible ideal sheaf~$\mathfrak I_{\mathfrak E}=
\pi^{*}(\mathfrak I_{\mathfrak D_1}+\mathfrak I_{\mathfrak D_2})$,
with corresponding Cartier divisor~$\mathfrak E$.
Since $\mathfrak D_1$ and $\mathfrak D_2$ coincide on the generic
fiber, $\mathfrak I_{\mathfrak D_1}+\mathfrak I_{\mathfrak D_2}$
is already invertible there and $\pi$ is an isomorphism
on the generic fiber.

The divisors $\pi^*\mathfrak D_1$
and $\pi^*\mathfrak D_2$ decompose canonically as sums
\[ \pi^*\mathfrak D_1=\mathfrak D'_1+\mathfrak E,
\quad  \pi^*\mathfrak D_2=\mathfrak D'_2+\mathfrak E.\]
Let us pose
\[ \mathfrak D'=\mathfrak D'_1+\mathfrak D'_2+\mathfrak E
=\mathfrak D'_1+\pi^*\mathfrak D_2=\pi^*\mathfrak D_1+\mathfrak D'_2. \]
An explicit computation on the blow-up shows that
$(\mathfrak X',\mathfrak D',e)$
and
$(\mathfrak X',\mathfrak E,e)$
are models of $\norm{\cdot}_{\min}$ and $\norm{\cdot}_{\max}$ respectively.
In particular, these metrics are smooth.

\emph{Assume that the initial metrics are semi-positive,
and that some positive power of~$L$ is effective. Then,
the metric~$\norm\cdot_{\min}$ is semi-positive too.}
By approximation, it suffices to treat the case where
the initial metrics are smooth and semi-positive.
Then, the previous construction applies. Keeping the introduced notation,
let us show that the restriction to the special fiber of
the divisor~$\mathfrak (D')_{\tilde K}$ is numerically effective.
Let $C\subset\mathfrak X'_{\tilde K}$ be an integral curve and
let us prove that $C\cdot\mathfrak( D')_{\tilde K}$ is nonnegative. If
$C$ is not contained in~$\mathfrak D'_1$, then $C\cdot(\mathfrak D'_1)_{\tilde K}\geq 0$,
and $C\cdot(\pi^*\mathfrak D_2)_{\tilde K}=\pi_*C\cdot\mathfrak D_2\geq 0$ since
$(\mathfrak D_2)_{\tilde K}$ is numerically effective; consequently,
$C\cdot\mathfrak (D')_{\tilde K}\geq 0$. Similarly, $C\cdot\mathfrak (D')_{\tilde K}\geq0$
when $C$ is not contained in~$\mathfrak D'_2$.
Since $\mathfrak D'_1\cap\mathfrak D'_2=\emptyset$, this shows
that $C\cdot\mathfrak D'_{\tilde K}\geq 0$ in any case, hence $(\mathfrak D')_{\tilde K}$
is numerically effective.

This last result  is the analogue in the ultrametric case
to the fact that the maximum of two continuous
plurisubharmonic functions is continuous plurisubharmonic.
However, observe that in the complex case, the maximum
or the minimum of smooth functions are not smooth in general.

\myparagraph{Measures (smooth metrics)}
In the non-archimedean case,
there isn't yet a purely  analytic incarnation of
the curvature form (or current) $c_1(\bar L)$
of a metrized line bundle~$\bar L$, although
the non-archimedean Arakelov geometry of~\cite{bloch-g-s95b}
should certainly be pushed forward in that direction.
However, as I discovered
in~\cite{chambert-loir2006}, one can define
an  analogue of the measure $c_1(\bar L)^n$ when the
space~$\mathrm X$ has dimension~$n$.

The idea consists in observing the local height pairing
(defined by arithmetic intersection theory) and
defining the measures so that a formula analogous to
the complex one holds.

Let us therefore consider
smooth metrized line bundles $\bar L_j $ (for $0\leq j\leq n$)
as well as regular meromorphic sections~$s_j$ which
have no common zero on~$X$.
There exists a proper model~$\mathfrak X$ of~$\mathrm X$ over~$K^\circ$
and, for each~$j$, a line bundle~$\mathfrak L_j$ on~$\mathfrak X$
which extends some power~$L_j^{e_j}$ of~$L_j$ and which defines
its metric.

Let $\mathrm Z\subset \mathrm X$ be an algebraic $k$-dimensional subvariety
and let $\mathfrak Z$ be its Zariski closure in~$\mathfrak X$;
this is a $(k+1)$-dimensional subscheme of~$\mathfrak X$.
Let's replace it by its normalization
or, more precisely, by its integral closure in its generic fiber.
The local height pairing is then given by intersection theory, as
\[  (\hdiv(s_0)\dots\hdiv(s_k)|\mathrm Z) =  (c_1(\div(s_0|_{\mathfrak Z}))
\dots c_1(\div(s_k|_{\mathfrak Z}))|\mathfrak Z) \, \log\abs\pi^{-1}, \]
where $\div(s_j|_{\mathfrak Z})$ means the divisor of~$s_j$,
viewed as a regular meromorphic section of~$\mathfrak L_j$
over~$\mathfrak Z$. The right hand side means taking
the intersection of the indicated Cartier divisors on~$\mathfrak Z$,
which is a well-defined class of a $0$-cycle supported by
the special fiber of~$\mathfrak Z$; then take its degree
and multiply it by $\log\abs\pi^{-1}$. (Recall that $\pi$ is a fixed
uniformizing element of~$K$; is absolute value does not depend on the
actual choice.)

When one views $s_k|_{\mathrm Z}$
as a regular meromorphic section of~$\mathfrak L_k$ on~$\mathfrak Z$
its divisor has two parts: the first one, say~$H$, is ``horizontal''
and is the Zariski closure of the divisor~$\div(s_k|_Z)$;
the second one, say~$V$, is vertical, \ie, lies in the special fiber
of~$\mathfrak Z$ over the residue field of~$K^\circ$.
This decomposes the local height pairing as a sum
\begin{align*}   (\hdiv(s_0)\dots\hdiv(s_k)|\mathrm Z)  \hskip -2cm\\
& =  (c_1(\div(s_0|_{\mathfrak Z})) \dots c_1(\div(s_k|_{\mathfrak Z}))|\mathfrak Z) \log\abs\pi^{-1} \\
&=  (c_1(\div(s_0|_{\mathfrak Z})) \dots c_1(\div(s_{k-1}|_{\mathfrak Z}))|
\div(s_k|_{\mathfrak Z}))\log\abs\pi^{-1} \\
&=  (c_1(\div(s_0|_{\mathfrak Z})) \dots c_1(\div(s_{k-1}|_{\mathfrak Z}))|H)
\log\abs\pi^{-1}  \\
& \qquad {}
+  (c_1(\div(s_0|_{\mathfrak Z})) \dots c_1(\div(s_{k-1}|_{\mathfrak Z}))|V)
\log\abs\pi^{-1} . \end{align*}
The first term is the local height pairing
of $\div(s_k|_{\mathrm Z})$. Let us investigate the second one.

Let $(V_i)$ be the family of irreducible components of
this special fiber; for each~$i$, let $m_i$ be its multiplicity
in the fiber.
Then, the vertical component~$V$ of~$\div(s_k|_{\mathfrak Z})$
decomposes as
\[ V = \sum_i c_i m_i V_i, \]
where $c_i$ is  nothing but the order of vanishing
of~$s_k$ along the special fiber at the generic point of~$V_i$.
Then,
\begin{multline*}
 (c_1(\div(s_0|_{\mathfrak Z})) \dots c_1(\div(s_{k-1}|_{\mathfrak Z}))|V) \\
= \sum_i c_i m_i (c_1(\div(s_0|_{\mathfrak Z})) \dots c_1(\div(s_{k-1}|_{\mathfrak Z}))|V_i).
\end{multline*}
Since $V_i$ lies within the special fiber of~$\mathfrak X$,
\[ (c_1(\div(s_0|_{\mathfrak Z})) \dots c_1(\div(s_{k-1}|_{\mathfrak Z}))|V_i)
= (c_1(\mathfrak L_0)\dots c_1(\mathfrak L_{k-1})|V_i) ,\]
the multidegree of the vertical component~$V_i$ with respect
to the restriction on the special fiber of the line bundles
$\mathfrak L_0,\dots,\mathfrak L_{k-1}$.

One remarkable aspect of Berkovich's theory is the existence,
for each~$i$, of a unique point~$v_i$ in~$\mathrm Z$ which specializes
to the generic point of~$V_i$. (Here, we use
that $\mathfrak Z$ is integrally closed in its generic fibre.)
Then,
\[ \log \norm {s_k}^{-1} (v_i) = c_i \log\abs\pi^{-1}. \]
Finally,
\begin{multline*}
 (c_1(\div(s_0|_{\mathfrak Z})) \dots c_1(\div(s_{k-1}|_{\mathfrak Z}))|V)\log\abs\pi^{-1}
 \\
= \sum_i \log \norm{s_k}^{-1}(v_i)  (c_1(\mathfrak L_0)\dots c_1(\mathfrak L_{k-1})|V_i) .
\end{multline*}

Let us sum up this calculation: we have introduced points~$v_i\in \mathrm Z$
and decomposed
the local height pairing as a sum:
\begin{multline*}
 (\hdiv(s_0)\dots\hdiv(s_k)|\mathrm Z)
= (\hdiv(s_0)\dots\hdiv(s_{k-1})|\div(s_k|_{\mathrm Z}))
\\ + \sum_i \log\norm{s_k}^{-1}(v_i) m_i (c_1(\mathfrak L_0)\dots c_1(\mathfrak L_{k-1})|V_i)
. \end{multline*}
It now remains to define
\begin{equation}
\label{eq.def-smoothmeasure}
 c_1(\bar L_0)\dots c_1(\bar L_{k-1})\delta _{\mathrm  Z}
 = \sum_i m_i (c_1(\mathfrak L_0)\dots c_1(\mathfrak L_{k-1})|V_i)
 \delta_{v_i} , \end{equation}
where $\delta_{v_i}$ is the Dirac measure at the point~$v_i\in {\mathrm Z}$.
This is a \emph{measure} on~$\mathrm X$, whose support
is contained in~$\mathrm Z$,
and whose total mass equals
\begin{align*}
 \sum_i (c_1(\mathfrak L_0)\dots c_1(\mathfrak L_{k-1})|V_i)
& = (c_1(\mathfrak L_0)\dots c_1(\mathfrak L_{k-1})|V) \\
& = (c_1(L_0)\dots c_1(L_{k-1})|\mathrm Z). \end{align*}
One can also check that it does not depend on the choice of
the section~$s_k$.

With this definition,
the local height pairing obeys an induction formula
totally analogous to the one satisfied in the complex case:
\begin{multline}\label{eq.inductive-u}
 (\hdiv(s_0)\dots\hdiv(s_k)|{\mathrm Z}) \\
{} =(\hdiv(s_0)\dots\hdiv(s_{k-1})|\div(s_k|_{\mathrm Z})) \\
{} + \int_{\mathrm X} \log\norm {s_k}^{-1} c_1(\bar L_0)\dots c_1(\bar L_{k-1})\delta_{\mathrm Z}.\end{multline}

\myparagraph{Local height pairing (admissible metrics)}
\label{sec.local.height.berk.ad}
With the notation of the previous paragraph,
observe that the measures we have defined are positive
when the smooth metrized line bundles are semi-positive.
Indeed, this means that the
line bundles~$\mathfrak L_j$ are numerically effective hence,
as a consequence of the criterion Nakai--Moishezon,
any subvariety of the special fiber has a nonnegative multidegree.

With basically the same argment that the one we sketched
in the complex case, we conclude that the local height pairing
extends by continuity when  semi-positive  metrized
line bundles are approximated by smooth semi-positive
metrized line bundles.  By linearity, this extends
the local height pairing to admissible metrized line bundles.

\myparagraph{Measures (admissible metrics)}
Let us now return to semi-positive metrized line bundles
$\bar L_0,\dots,\bar L_{k-1}$, approximated
by smooth semi-positive metrized line bundles $\bar L_j^{(m)}$.
I claim that for any $k$-dimensional variety~$\mathrm Z\subset\mathrm X$,
the measures $c_1(\bar L_0^{(m)})\dots c_1(\bar L_{k-1}^{(m)})\delta_{\mathrm Z}$ converge to a measure on~$\mathrm X$.

To prove the claim, we may assume that $\bar L_0,\dots,\bar L_{k-1}$
have sections $s_0,\dots,s_{k-1}$ without common zeroes on~$\mathrm Z$.
Let also consider a smooth function~$\phi$
on~$\mathrm X$; let $\bar L_k$ be the trivial line bundle
with the section $s_k=1$, metrized in such a way that $\norm{s_k}=e^{-\phi}$.
Then, one has
\[ \int_{\mathrm X}\phi c_1(\bar L^{(m)}_0)\dots c_1(\bar L^{(m)}_{k-1})\delta_{\mathrm Z}
= (\hdiv(s_0)^{(m)}\dots\hdiv(s_{k-1})^{(m)}\hdiv(s_k)|\mathrm Z);
\]
writing $\bar L_k$ has the quotient of two ample metrized line bundles,
we deduce from the existence of the local height pairing for admissible
metrics that these integrals converge when $m\ra\infty$.
Consequently, the sequence of measures
$(c_1(\bar L^{(m)}_0)\dots c_1(\bar L^{(m)}_{k-1})\delta_{\mathrm Z} )_m
$
converges to a positive linear form on the space of smooth functions.
By a theorem of Gubler (\cite{gubler1998}, Theorem~7.12),
which builds on the Stone-Weierstraß theorem
and the compactness of the Berkovich space~$\mathrm X$,
the space of smooth functions is dense
in the space of continuous complex functions on~$\mathrm X$.
A positivity argument, analogous to the proof that
positive distributions are measures,
then implies that  our linear form is actually a positive measure
which deserves the notation
\[ c_1(\bar L_0)\dots c_1(\bar L_{k-1})\delta _{\mathrm Z}. \]

We then extend this definition by linearity to the
case of arbitrary admissible line bundles. The total mass
of this measure is again the multidegree of~${\mathrm Z}$ with respect
to the line bundles~$L_j$ (for $0\leq j\leq k-1$).

\myparagraph{Integrating Green functions}\label{sec.mahler}
The definition of the convergence of a sequence of measures
is convergence of all integrals against a given \emph{continuous}
compactly supported function. In applications, however, it can be
desirable to integrate against more general functions.
The inductive formula~\eqref{eq.inductive-C}
for the local height pairing in the complex case,
is such an example, as is the interpretation of Mahler measures
of polynomials as (the archimedean component of) heights.
However, its analogue (Equation~\ref{eq.inductive-u})
a priori holds only when $\log\norm{s_0}^{-1}$ is continuous,
that is when the section~$s_0$ has no zeroes nor poles.

The fact that it still holds in the archimedean case
is a theorem of Maillot~\cite{maillot2000}
building on the theory of Bedford--Taylor.
We proved in~\cite[Th.~4.1]{chambert-loir-thuillier2009}
that this relation holds in the ultrametric case too.
The proof (valid both in the ultrametric and archimedean cases)
works by induction,
and ultimately relies on an approximation lemma according
to which any semi-positive Green function~$g$ for a divisor~$D$
is an increasing limit of smooth functions~$(g_n)$
such that, for any~$n$, $g-g_n$ is a semi-positive Green function for~$D$.
In fact, it suffices to pose $g_n=\min(g,n\log\abs\pi^{-1})$;
then, $g-g_n=\max(0,g-n\log\abs\pi^{-1})$ is the maximum of
two semi-positive Green functions, hence is semi-positive.
(In the archimedean case, one needs to further regularize~$g_n$;
see~\cite{chambert-loir-thuillier2009} for details.)

The symmetry of the local height pairing then implies the following
analogue of the Poincar\'e--Lelong formula.
When $\bar L$ is the trivial line bundle, with
the metric defined by an admissible function~$\phi$,
the factor $c_1(\bar L)$ will be written $\ddc\phi$, by analogy
to the complex case.
\begin{CLprop}\label{prop.pl}
Let $\phi$ be a smooth function on~$\mathrm X$ and let $\bar L_1,
\dots,\bar L_k$ be admissible metrized line bundles;
let $\mathrm Z$ be a $k$-dimensional subvariety of~$\mathrm X$
and let $s$ be an invertible meromorphic sections of~$\bar L_1$.
Then,
\begin{multline*}
 \int_{\mathrm X} \phi c_1(\bar L_1)\dots c_1(\bar L_k)\delta_{\mathrm Z}
\\= \int_{\mathrm X} \phi c_1(\bar L_2)\dots c_1(\bar L_k)\delta_{\div(s|\mathrm Z)}
+ \int_X \log\norm s^{-1} \ddc\phi  c_1(\bar L_2)\dots c_1(\bar L_k) \delta_{\mathrm Z}
.\end{multline*}
\end{CLprop}
\begin{proof}
Let $\bar L_0$ be the trivial line bundle with global section $s_0=1$
and metric defined by $\phi=\log\norm{s_0}^{-1}$. Let $s_1=s$ and,
for $2\leq j\leq k$, let $s_j$ be an invertible meromorphic section
of~$L_j$.
Since $\div(s_0|_{\mathrm Z})=0$,
\[ (\hdiv(s_0)\dots\hdiv(s_k)|\mathrm Z)
= \int_{\mathrm X} \phi c_1(\bar L_1)\dots c_1(\bar L_k)\delta_{\mathrm Z}\]
and
\[ (\hdiv(s_0)\hdiv(s_2)\dots\hdiv(s_k)|\div(s|_{\mathrm Z}))
= \int_{\mathrm X} \phi c_1(\bar L_2)\dots c_1(\bar L_k)\delta_{\div(s|\mathrm Z)}\]
One the other hand, the symmetry of the local height pairing implies that
\begin{align*}
 (\hdiv(s_0)\dots\hdiv(s_k)|\mathrm Z) &=  (\hdiv(s_1)\hdiv(s_0)\dots\hdiv(s_k)|\mathrm Z)  \\
&= (\hdiv(s_0)\hdiv(s_2)\dots\hdiv(s_k)|\div(s|_{\mathrm Z})) \\
& \qquad + \int_{\mathrm X} \log\norm{s}^{-1} c_1(\bar L_0)c_1(\bar L_2)\dots c_1(\bar L_k)\delta_{\mathrm Z}.
\end{align*}
Combining these equations, we obtain the claim.
\end{proof}

\section{Examples}
In this section,  we give some examples of metrics and measures.
Without mention of the contrary, we stick to the non-archimedean case;
basic notation concerning~$K$, $K^\circ$, etc.,
is as in Section~\ref{sec:meas-na}.

\subsection{The projective space}
Let $\mathrm X$ be the projective space~$\mathrm P^n_K$
and let $\mathscr O(1)$ be the tautological line bundle
on~$\mathrm X$, together with its Weil metric.
Let us describe the associated measure, taking the opportunity
to add details concerning Berkovich spaces.

As we remarked above, the Weil metric is induced by the tautological line
bundle on the projective scheme~$\mathfrak X=\mathbf P^n_{K^\circ}$
and is smooth.
The special fiber of~$\mathfrak X$ is the projective space~$\mathbf P^n_{\tilde K}$
over the residue field of~$K^\circ$; it is in particular
irreducible. Moreover, the degree of the
tautological line bundle is equal to~$1$.
The measure $c_1(\overline{\mathscr O(1)})^n$ is therefore equal
to the Dirac mass at the unique point of~$\mathrm X$
which reduces to the generic point of the special fiber.
It remains to describe this point more precisely.

The scheme~$\mathbf P^n_{K^\circ}$ is the union of $(n+1)$ affine
open subsets $\mathfrak U_0,\dots,\mathfrak U_n$
defined by the non-vanishing of the homogeneous coordinates
$x_0,\dots,x_n$. Their generic fibers in the sense of analytic
geometry are
$n+1$ affinoid subsets $\mathrm U_0,\dots,\mathrm U_n$,
which cover~$\mathrm P^n_K$.
In fact, $\mathrm U_i$ corresponds to the set of points~$[x_0:\dots:x_n]$
of~$\mathrm P^n_K$ such that $\abs{x_i}=\max(\abs{x_0},\dots,\abs{x_n})$.

To fix ideas, let us consider~$i=0$. Then,
$\mathfrak U_0=\Spec(K^\circ [T_1,\dots,T_n])$
is the affine space over~$K^\circ$ with coordinates~$T_j=x_j/x_0$.
The natural topology on the algebra $K^\circ[T_1,\dots,T_n]$,
and on its tensor product with~$K$, $K[T_1,\dots,T_n]$,
is induced by the \emph{Gauß norm}
\[ \norm{f} = \max_{\mathbf a\in\N^n} \abs{f_{\mathbf a}},
 \qquad f=\sum f_{\mathbf a} T_1^{a_1}\dots T_n^{a_n}. \]
The completion of~$K[T_1,\dots,T_n]$ for this norm
is the Tate algebra
and is denoted $K\langle T_1,\dots,T_n\rangle$.
It consists of all power series $f=\sum f_{\mathbf a}
 T_1^{a_1}\dots T_n^{a_n}$ with coefficients in~$K$
such that $\abs{f_{\mathbf a}}\ra 0$ when $\abs{\mathbf a}=a_1+\dots+a_n
\ra\infty$; it is endowed with the natural extension
of the Gauß norm, and is complete.
By definition, the generic fiber~$\mathrm U_0$ of~$\mathfrak U_0$
in the sense of analytic geometry is
the Berkovich spectrum of the Tate algebra, that is
the set of all multiplicative semi-norms on it
which are continuous with respect to the topology defined
by the Gauß norm.
Since the theorem of Gauß asserts that this norm is multiplicative,
it defines a point~$\gamma\in \mathrm U_0$, which we like to call
the Gauß point.

The reduction map $\mathrm U_0\ra \mathfrak U_0\otimes {\tilde K}$ is defined
as follows.
Let $x\in \mathrm U_0$,
let $\mathfrak p_x\subset K\langle T_1,\dots,T_d\rangle$
be the kernel of the semi-norm~$x$,
which is also the kernel of the canonical
morphism $\theta_x\colon K\langle T_1,\dots,T_d\rangle\ra \mathscr H(x)$.
The images $T_j(x)$ of the indeterminates~$T_j$
are elements of absolute value~$\leq 1$ of the complete
ultrametric field~$\mathscr H(x)$; they belong
to its valuation ring~$\mathscr H(x)^\circ$.
Letting $\widetilde{\mathscr H(x)}$ to be the residue field,
there exists a unique morphism~$\theta_{\bar x}\colon
{\tilde K}[T_1,\dots,T_d]\ra\widetilde{\mathscr H(x)}$ such that
$\theta_{\bar x}(T_i)$ is the image in~$\widetilde{\mathscr H(x)}$
of $T_i(x)$. The kernel of this morphism is
a prime ideal of the ring~${\tilde K}[T_1,\dots,T_d]$
and defines a point~$\bar x$ in the scheme~$\mathfrak U_0\otimes {\tilde K}$.

Let us now compute the reduction of the Gauß point~$\gamma$.
By definition, the field~$\mathscr H(\gamma)$ is the completion of
the Tate algebra $K\langle T_1,\dots,T_d\rangle$ for the
Gauß norm. I claim that morphism~$\theta_{\bar\gamma}$
is injective, in other words, that the images
of $T_1(\gamma),\dots,T_d(\gamma)$ in the residue field
$\widetilde{\mathscr H(\gamma)}$ are algebraically independent.
Let $P\in K^\circ[T_1,\dots,T_d]$ be any polynomial
whose reduction~$\bar P$ belongs to the kernel of~$\theta_{\bar\gamma}$;
this means $\abs{P}_\gamma<1$; in other words,
the Gauß norm of~$P$ is~$<1$ and each coefficient of~$P$ has
absolute value~$<1$. Consequently, $\bar P=0$ and $\theta_{\bar\gamma}$
is injective, as claimed.
This shows that $\bar \gamma$
is the generic point of the scheme~$\mathfrak U_0\otimes {\tilde K}$.

We thus have proved the following proposition.
\begin{CLprop}
The measure $c_1(\overline{\mathscr O(1)}_\W)^n$ on~$\mathrm P^n_K$
is the Dirac measure at the Gauß point~$\gamma$.
\end{CLprop}

\subsection{Semi-stable curves and reduction graphs}

In this section, we assume that $\mathrm X$ is the
analytic space associated to a projective curve over
a field~$K$ which is complete for  a discrete valuation.
The semi-stable reduction theorem of Deligne--Mumford
asserts that, up to replacing the base field~$K$
by a finite extension, the curve~$\mathrm X$ has
a projective model~$\mathfrak X$ over~$K^\circ$
which is regular (as a 2-dimensional scheme)
and whose special fiber is reduced, with at most double points
for singularities. We may also assume that
the irreducible components are geometrically irreducible.
We do not require, however, that $\mathfrak X$ is the
minimal semi-stable model.

\myparagraph{The reduction graph of the special fiber}
In that situation, the \emph{reduction graph}~$R(\mathfrak X)$
is a metrized graph defined as follows. It has for vertices the
irreducible components of the special fiber, with as many edges
of length~$\log\abs\pi^{-1}$
between two vertices as the number of intersection
points of the corresponding components.
In a neighborhood of a double point,  $\mathfrak X$
looks like (\ie, has an étale map to)
the scheme with equation $xy=\pi$ in the affine
plane~$\mathbf A^2_{K^\circ}$.

If one replaces the field~$K$ by a finite extension~$K'$,
the base change $\mathfrak X\otimes_{K^\circ} (K')^\circ$
may no longer be regular.  Indeed, $\mathfrak X\otimes_{K^\circ}(K')^\circ$
is étale locally isomorphic to $xy=(\pi')^e$, where
$\pi'$ is a uniformizing element of~$K'$, and $e$
is the ramification index. When $e>1$,
the origin is a singular point of that scheme and one needs
to blow it up repeatedly in order to obtain a regular scheme,
which is a semi-stable model of~$\mathrm X_{K'}$ over~$(K')^\circ$.
The two initial components are replaced by a chain
of $e+1$ components, the $e-1$ intermediate ones being
projective lines. In other words, $e-1$ vertices
have been added, regularly spaced along each edge.
One concludes that the reduction graph
has not changed, as a topological space.
Its metric has not changed either, since the $e$
edges that partition an original edge (of length $\log\abs\pi^{-1}$)
have length~$\log\abs{\pi'}^{-1}=\frac1e \log\abs\pi^{-1}$.

We say that a function on~$R(\mathfrak X)$
is piecewise linear if, up to passing to a finite extension
(which replaces each edge by $e$ edges of length equal to $1/e$th
of the initial one),
it is linear on each edge.

\myparagraph{Drawing the reduction graph on the Berkovich space}
Let us analyse the situation from the Berkovich viewpoint.
As we have seen, the generic points of the special fiber
are the reductions of canonical points of~$\mathrm X$: the vertices
of the graph~$R(\mathfrak X)$ naturally live in~$\mathrm X$.
The same holds for the edges, but is a bit more subtle.
As we have seen, blowing-up intersection points of components
in the special fiber gives rise to new components, hence
to new points of~$\mathrm X$. Would we enlarge the ground
field and blow-up indefinitely, the constellation of points in~$\mathrm X$ that
we draw converges to a graph which is isomorphic to~$R(\mathfrak X)$.

According to Berkovich~\cite{berkovich1999},
a far more precise result holds.
Let us consider a neighborhood~$\mathfrak U$
of a singular point of the special fiber, pretending
it is isomorphic to the locus defined by the equation $xy-\pi$
in~$\mathbf A^2$; so $\mathfrak U=\Spec( K^\circ [ x,y ]/(xy-\pi))$.
Its generic fibre is the affinoid space~$\mathrm U$ defined
by the equality $\abs{xy}=\abs\pi$ in the unit polydisk
$\mathrm B^2=\mathscr M(K\langle x,y\rangle) $.
The affinoid algebra of~$\mathrm U$ is the quotient
\[ K\langle x,y\rangle/(xy-\pi) \]
whose elements~$f$ are (non-uniquely) represented by a series
\[ \sum_{m,n=0}^\infty a_{m,n} x^m y^n, \]
with $a_{m,n}\ra 0$ when $m+n\ra\infty$. However, observing
that $x$ is invertible in this algebra, with inverse $\pi^{-1}y$,
so that $y=\pi x^{-1}$, we can replace
each product $x^my^n$ by $\pi^{n} x^{m-n}$, leading to
an expression of the form
\[ f=\sum_{n\in\Z} a_n x^n, \]
where $\abs{a_n}\ra 0$ when $n\ra +\infty$ and $\abs{a_n}\pi^{-n}\ra 0$
when $n\ra -\infty$. Such an expression is now unique, and is called
the Laurent expansion of~$f$.

It leads to a natural family $(\gamma_r)_{r\in[0,\log\abs\pi^{-1}]}$
of multiplicative seminorms on the algebra $\mathscr O(\mathrm U)$,
parametrized by the unit interval in~$\R$.
Namely, for each real number~$r\in[0,\log\abs\pi^{-1}]$, we can set
\[ \gamma_r (f)  =  \max_{n\in\Z} \abs{a_n}e^{-rn},
    \qquad f=\sum_{n\in\Z}a_n x^n\in\mathscr O(\mathrm U). \]
Obviously, $\gamma_r$ is a norm on~$\mathscr O(\mathrm U)$
which extends the absolute value of~$K$; its multiplicativity
is proved analogously to that of the Gauß norm.
It is easy to check that the map $[0,\log\abs\pi^{-1}]\ra \mathrm U$
defined by $r\mapsto \gamma_r$ is continuous
(this amounts to the fact that the maps $r\mapsto \gamma_r(f)$
are continuous), hence defines
an parametrized path in the topological space~$\mathrm U$.

Let $S(\mathfrak U)$ be its image (with the induced distance);
Berkovich calls it the skeleton of the formal scheme obtained
by completing~$\mathfrak U$ along its special fibre.
A point~$u$ in~$\mathrm U$ has two coordinates $(x(u),y(u))$
in the completed residue field~$\mathscr H(u)$
which are elements of absolute value~$\leq 1$ satisfying $x(u) y(u)=\pi$.
In particular,
\[ r(u) = \log\abs{x(u)}^{-1}  \in [0,\log\abs\pi^{-1}]. \]
The map $\rho\colon u\mapsto \gamma_{r(u)}$ is a continuous
function from~$\mathrm U$
to $S(\mathfrak U)$.

Let us compute the image of~$\gamma_r$ by this map.
By definition of~$\gamma_r$, one has
\[  \abs{x(\gamma_r)} =\gamma_r(x)= e^{-r}, \]
hence $r(\gamma_r)=r$ and $\rho(\gamma_r)=\gamma_r$.
In other words, the map $\rho$ is a retraction of~$\mathrm U$
onto the skeleton~$S(\mathfrak U)$.

The special fiber of~$\mathfrak U$ is defined by the
equation $xy=0$ in~$\mathbf A^2_{\tilde K}$, hence has two components.
One can check that the point $\gamma_0$ reduces to the generic
point of the component with equation~$y=0$, while $\gamma_{\log\abs\pi^{-1}}$
reduces to the generic point of the component with equation~$x=0$.

\medskip

These constructions have to be done around each singular point
of the special fiber of~$\mathfrak X$, locally for the étale
topology of~$\mathfrak X$.
Berkovich proves that they can be glued,
so that the graph $R(\mathfrak X)$ is again canonically
interpreted as an actual metrized graph drawn
on the analytic space~$\mathrm X$; we write $\iota\colon R(\mathfrak X)\hra\mathrm X$ for the canonical embedding.
The map~$\iota$ admits
a continuous retraction $\rho\colon\mathrm X\ra R(\mathfrak X)$.

Although we will not use this fact, we must mention
that the retraction~$\rho$ is a deformation retraction.

\myparagraph{Metrized line bundles and the reduction graph}
A construction of S.~Zhang~\cite{zhang93}, building
on prior results of Chinburg--Rumely~\cite{chinburg-r1993},
furnishes continuous metrics on divisors from continuous functions
on the reduction graph~$R(\mathfrak X)$.
It works as follows. First of all, if $P\in \mathrm X(K)$
is a rational point, there is a unique
morphism $\eps_P\colon\Spec K^\circ\ra\mathfrak X$
which extends the point~$P$ viewed as a morphism from~$\Spec K$
to~$X$. The image of this section is a divisor~$D_P$
on~$\mathfrak X$ and the line bundle $\mathscr O(D_P)$ on~$\mathfrak X$
defines a smooth metric on~$\mathscr O(P)$;
we write $\overline{\mathscr O(P)}_{\mathfrak X}$ for the corresponding
metrized line bundle.
 We also define $\mu_P$ as the Dirac measure at the vertex of the graph
corresponding to the (unique) irreducible component of
the special fiber by which $D_P$ passes through.
The construction and the notation is extended by additivity
for divisors which are sums of rational points.
More generally, if $P$ is only a closed point of~$X$, we do this construction
after the finite extension $K(P)/K$, so that $P$ becomes a sum
of rational points, using for model the minimal resolution
of $\mathfrak X\otimes K(P)^\circ$ described earlier.

If $f$ is any continuous function on~$R(\mathfrak X)$
and $D$ a divisor on~$\mathrm X$, the metrized line bundle $\mathscr O(D+f)_{\mathfrak X}$
is deduced from $\overline{\mathscr O(D)}_{\mathfrak X}$
by multiplying the metric by $e^{-f}$.
When $f$ is piecewise linear, this metrized line bundle
is smooth. To prove that, we may extend the scalars
and assume that $D$ is a sum of rational points $\sum n_j P_j$
and that $f$ is linear on each edge corresponding to
an intersection point of components of the special fiber.
Letting $(V_i)$ being the family of these components,
and writing~$v_i$ for the vertex of~$R(\mathfrak X)$
corresponding to~$V_i$,
the divisor
\begin{equation}
\label{eq.model}
 \sum_j n_j D_{P_j} + \sum_i f(v_i) V_i
\end{equation}
defines the metrized line bundle $\mathscr O(D+f)_{\mathfrak X}$.

In this context, Zhang has defined a curvature operator,
which associates to a metrized line bundle
a distribution on the graph~$R(\mathfrak X)$, defined
in such a way that
\begin{itemize}
\item for any divisor~$D$ on~$\mathrm X$,
$\curv(\overline{\mathscr O(D)}_{\mathfrak X})=\mu_D$;
\item for any continuous function~$f$,
$\curv(\overline O(f)_{\mathfrak X})=-\Delta f$, where $\Delta$ is
the Laplacian operator of the graph~$R(\mathfrak X)$,
\end{itemize}
and depending linearly on the metrized line bundle.
The following lemma compares this construction with
the general one on Berkovich spaces.
\begin{CLlemm}
Let $\bar L=\overline{\mathscr O(D+f)}_{\mathfrak X}$ be a metrized
line bundle on~$\mathrm X$ associated to a divisor~$D$ on~$\mathrm X$
and a continuous function~$f$ on the graph~$R(\mathfrak X)$.
If it is semi-positive, resp.~admissible in the sense of~\cite{zhang93},
then it is semi-positive, resp.~admissible in the sense of this article, and
one has
\[ c_1(\bar L)=\iota_* \curv(\overline{\mathscr O(D+f)}) \log\abs\pi^{-1}.\]
\end{CLlemm}
In other words, the measure $c_1(\bar L)$
is supported by the graph~$R(\mathfrak X)$ where
it coincides essentially with Zhang's curvature.
\begin{proof}
We first assume that $f$ is linear on each edge of~$R(\mathscr X)$
and that $D$ is a sum of rational points of~$X$.
Then, $\bar L$ corresponds to the line bundle~$\mathfrak L$ on
the model~$\mathfrak X$ given by Equation~\ref{eq.model}.
By definition, the measure $c_1(\mathfrak L)$ is computed as follows.
It is a sum, for all components~$V_i$ of the special fiber,
of $\deg(\mathfrak L|V_i)\log\abs\pi^{-1}$
times the Dirac measure at the corresponding point~$v_i$
of~$R(\mathfrak X)$. In particular, it is supported by~$R(\mathfrak X)$.
Then,
\[ \deg(\mathfrak L|V_i)
= \sum_j n_j \left\{\begin{array}{cc}
         1 & \text{if $D_{P_j}$ passes through $V_i$;}\\
        0 & \text{otherwise}\end{array} \right\}
+ \sum_j f(V_j) (V_i,V_j), \]
where $(V_i,V_j)$ is the intersection number of the divisors~$V_i$
and~$V_j$. That $D_{P_j}$ passes through~$V_i$
means exactly that $\rho(P_j)=v_i$.
Moreover, if $j\neq i$, then $(V_i,V_j)=m_{i,j}$ is just the
number of  intersection points of~$V_i$ and~$V_j$,
while
\[ (V_i,V_i)=(V_i,\sum_j V_j) - \sum_{j\neq i} (V_i,V_j)
 = -\sum_{j\neq i} (V_i,V_j), \]
since the whole special fiber is numerically equivalent to zero.
Consequently,
\[ \sum_j f(v_j) (V_i,V_j)
   = \sum_{j\neq i} m_{i,j} \big(f(V_j) - f(V_i)\big). \]
Observe that this is the sum, over all edges from~$V_i$,
of the derivative of~$f$ along this edge.
Comparing with the definitions given by Zhang in~\cite{zhang93},
one finds, for any function~$g$ on~$R(\mathfrak X)$
\begin{align*}
\sum_i \deg(\mathfrak L|V_i) g(v_i)
& = \sum_{j}  n_j  g(\rho(P_j))
 +\sum_i \langle \delta f(v_i), g\rangle \\
&= \int_{R(\mathfrak X)} g \, (\mu_{D} +\delta f)\\
& =\int_{R(\mathfrak X)} g\, \curv(\overline{\mathscr O(D+f)}). \end{align*}
This proves the claimed formula when $f$ is linear on each
edge of~$\mathfrak X$ and $D$ is a sum of rational points.

By working over an appropriate finite extension of~$K$,
it extends to the case where $f$ is only piecewise linear,
$D$ being any divisor on~$X$.

Zhang defines $\overline{\mathscr O(D+f)}_{\mathfrak X}$
to be semi-positive
if $f$ is a uniform limit of piecewise linear functions~$f_n$
such that $\curv(\overline{\mathscr O(D+f_n)}_{\mathfrak X})\geq 0$.
The metrized line bundle~$\bar L$ is then the limit of
the metrized line bundles~$\bar L_n$
corresponding to models $\mathfrak L_n$ (on appropriate
models~$\mathfrak X_n$ of~$\mathrm X$ after some extension
of scalars) of~$\mathscr O(D)$. By the previous
computation, these metrics are smooth and $c_1(\bar L_n)\geq 0$.
Reversing the computation, this means that $\mathfrak L_n$
is numerically effective on~$\mathfrak X_n$, hence
$\bar L$ is semi-positive.
By definition of the measure~$c_1(\bar L)$, one has
\begin{align*}
 c_1(\bar L) & =\lim_n c_1(\bar L_n)
 =\lim_n \iota_* \curv(\overline{\mathscr O(D+f_n)})\\
& =\iota_* \lim_n \curv(\overline{\mathscr O(D+f_n)})
=\iota_* \curv(\overline{\mathscr O(D+f)}). \end{align*}

The case of an admissible metrized line bundle follows by linearity.
\end{proof}

\subsection{Local character of the measures}
The definition of the measures associated to metrized line
bundles is global in nature.
Still, the main result of this section implies
that they are local.

\begin{CLdefi}
Let $\mathrm X$ be an analytic space. A function on~$\mathrm X$
is said to be \emph{strongly pluriharmonic} if it is
locally a uniform limit of functions of the form $a\log\abs u$,
where $a\in\R$ and $u$ is holomorphic and nonvanishing.
\end{CLdefi}

There is a general theory of harmonic functions on curves
due to Thuillier~\cite{thuillier2005}
(see also~\cite{favre-rl2007,baker-rumely2010} on the projective line;
note that the definition of a strongly harmonic function of the latter reference
is different from the one adopted here).
Strongly pluriharmonic functions are harmonic in their sense.
Indeed, logarithms of absolute values of invertible holomorphic
functions are harmonic, and harmonic functions are preserved
by uniform limits
(Prop.~2.3.20 and~3.1.2 of~\cite{thuillier2005}).
In fact, when the residue field of~$K$ is algebraic over a finite
field, any harmonic function is locally of the form
$a\log\abs u$, where $a\in\R$ and $u$ is 
an invertible holomorphic function
(\loccit, Theorem~2.3.21).

This is not necessarily the case for more general fields~$K$:
there are harmonic functions over analytic curves
which are not locally equal to the logarithm of the absolute value
of an invertible function;
examples require to consider curves of genus~$\geq  1$.
In a conversation with A.~Ducros, we devised the following example
of a one-dimensional affinoid space. Let $\mathfrak E$ be
an elliptic scheme over~$K^\circ$, let $o$ be the origin in~$\mathfrak E_{\tilde K}$ and let $p$ be a \emph{non-torsion} rational point
in~$\mathfrak E_{\tilde K}$; let $\mathfrak X$ be the blow-up
of~$\mathfrak E$ at the point~$p$. Let then $\mathfrak U$ be its
open subset obtained by removing the point~$o$ as well as
a smooth point in the exceptional divisor of the blow-up;
its generic fiber~$\mathrm U$ is the desired affinoid space
--- it is the complementary subset in the elliptic curve~$\mathrm E_K$
to two small disjoint disks.
One can prove that the space of harmonic functions
on~$\mathrm U$ is 2-dimensional, and that all holomorphic invertible
functions on~$\mathrm U$ have constant absolute value.

I do not know whether any harmonic function on a curve is \emph{locally}
a uniform limit of logarithms.

\begin{CLdefi}
Let $\bar L$ be a metrized line bundle
on an analytic space~$\mathrm X$ and let $\mathrm U$
be an open subset of~$\mathrm X$.  One says
that $\bar L$ is strongly pluriharmonic on~$\mathrm U$
if for any local frame~$s$ of~$L$
defined on an open subset~$\mathrm V\subset\mathrm U$,
$\log\norm s^{-1}$ is strongly pluriharmonic on~$\mathrm V$.
\end{CLdefi}
A metrized line bundle is strongly pluriharmonic on~$\mathrm U$
if it admits, in a neighbourhood of any point of~$\mathrm U$,
a local frame whose norm is identically equal to~$1$.
(For the converse to hold, one would need to introduce
a notion of local frame for real line bundles.)

\begin{CLprop}\label{prop.local}
Let $\mathrm X$ a the analytic space associated to a proper $K$-scheme.
Let $\bar L_1,\bar L_2,\dots,\bar L_k$ be admissible metrized line bundles
on~$\mathrm X$. Let $\mathrm Z$ be a $k$-dimensional
Zariski closed subset of~$\mathrm X$.
Assume that $\bar L_1$ is strongly pluriharmonic on~$\mathrm U$.
Then, the support of the measure
$c_1(\bar L_1)\dots c_1(\bar L_k)\delta_{\mathrm Z}$
is disjoint from~$\mathrm U$.
\end{CLprop}
\begin{proof}
One has to show that for any continuous function~$\phi$
with compact support contained in~$\mathrm U$
\[ \int_{\mathrm X} \phi \, c_1(\bar L_1)\dots c_1(\bar L_k)\delta_{\mathrm Z}=0.\]
By Gubler's theorem, the space of smooth functions is dense
in the space of continuous functions on~$\mathrm X$.
Using the fact that the maximum and the minimum of smooth functions
are still smooth,
one proves that the space of smooth functions with compact support
contained in~$\mathrm U$ is dense in the space of continuous functions
with compact support contained in~$\mathrm U$, for the topology
of uniform convergence.
We thus may assume that $\phi$ is smooth, with compact support
contained in~$\mathrm U$. Finally, we may also assume that
the metric on the line bundles~$\bar L_2,\dots,\bar L_k$ are smooth.

We may argue locally and assume that $L_1$ has a meromorphic
section~$s$ whose divisor $\div(s)$ is disjoint from~$\mathrm U$.
Up to shrinking~$\mathrm U$ again, we may assume that
there exists a sequence $(u_n)$ of rational
functions without zeroes nor poles on~$\mathrm U$ such that
$\log\norm s=\lim \log \abs{u_n}^{1/n}$.


According to Prop.~\ref{prop.pl}, one has
\begin{multline*}
\int_{\mathrm X} \phi\,c_1(\bar L_1) \dots c_1(\bar L_k) \delta_{\mathrm Z} \\
=
\int_{\mathrm X} \phi c_1(\bar L_2)\dots c_1(\bar L_k)\delta_{\div(s|_{\mathrm Z})}
 + \int_{\mathrm X} \log\norm s^{-1} \ddc\phi c_1(\bar L_2)\dots c_1(\bar L_k) \delta_{\mathrm Z} ,\end{multline*}
The first term vanishes because $\div(s|_{\mathrm Z})$
and the support of~$\phi$ are disjoint.
The second is the limit of
\[  \int_{\mathrm X}  \log\abs {u_n}^{-1/n}
    \ddc\phi c_1(\bar L)^{k-1}\delta_{\mathrm Z} . \]
Using the fact that $\div(u_n)\cap\mathrm U$
is empty and applying the same computation, the term of index~$n$ equals
\[
\frac1n \int_{\mathrm X} \phi \, c_1(\bar M_n)  c_1(\bar L_2)\dots c_1(\bar L_k)\delta_{\mathrm Z}. \]
where $\bar M_n$ is the trivial metrized line bundle~$\mathscr O_X$,
and its meromorphic section~$u_n$ replacing~$s$.
But this integral is zero, by the formula~\eqref{eq.def-smoothmeasure}
which defines the measure associated to smooth metrized line bundles.
\end{proof}

\subsection{Polarized dynamical systems}
We now explain another example of metrized line bundles:
the canonical metric associated to a dynamical system.

\begin{CLlemm}\label{lemm.tate}
Let $\mathrm X$ be the analytic space associated to a proper $K$-scheme
and let $f\colon\mathrm X\ra\mathrm X$ be a finite
morphism. Let $L$ be a line bundle on~$\mathrm X$,
$d$ an integer such that $d\geq 2$ and an isomorphism
$\eps\colon f^*L\simeq L^d$.
The line bundle $L$ possesses a unique continuous metric
such that the isomorphism~$\eps$ is an isometry.
If $L$ is ample, then this metric is semi-positive.
\end{CLlemm}

In essence, this result, or at least its proof, goes back
to Tate's construction of the ``Néron--Tate'' canonical height for abelian
varieties. In the slightly different language of local heights
and N\'eron functions,
it has been proved by Call--Silverman~\cite{call-s93}.
In the asserted form, it is due to Zhang~\cite{zhang95b}.

\begin{proof}
Let us first prove uniqueness. If $\bar L$ and~$\bar L'$
are two metrics on~$L$,
let $\phi$ be the continuous
function such that $\norm\cdot'=e^{-\phi}\norm\cdot $.
Assuming that $\eps$ is an isometry for these two metrics,
one obtains the following equation
\[ \phi(f(x))= d \phi(x), \]
for any $x\in \mathrm X$. Since $\mathrm X$ is compact,
$\phi$ is bounded and this equation implies
that $\norm\phi_\infty\leq \frac1d \norm\phi_\infty$.
Since $d\geq 2$, one concludes that $\phi\equiv 0$.

For the existence, one begins with any continuous
metric~$\bar L_0$ on~$L$.
Let us then consider the sequence of metrics~$(\bar L_n)$ on~$L$
induced by the pull-backs on $L^{d}=\eps f^*L$,
$L^{d^2}=(\eps f^*)^2L$, etc., hence on~$L$.
Since $d\geq 2$, a similar contraction argument as the
one used for uniqueness shows that
this is a Cauchy sequence of metrics on~$L$; consequently, it converges
to a continuous metric on~$L$. If $\bar L_0$ is chosen to be semi-positive,
which we may if $L$ is ample, then all al of the metrized line bundles~$\bar L_n$ are semi-positive, hence the canonical metric is semi-positive.

Concretely, in the non-archimedean case, one begins with a
model $(\mathfrak X_0,\mathfrak L_0,e)$
such that $\mathfrak L_0$ is numerically effective.
Then one considers the map $f\colon X\ra\mathfrak X_0$
and the normalization~$\mathfrak X_1$ of~$\mathfrak X_0$ in~$X$;
this is a projective model~$\mathfrak X_1$, equiped with a finite
morphism $f_1\colon \mathfrak X_1\ra\mathfrak X_0$ extending~$f$.
Moreover, $\mathfrak L_1=f_1^*\mathfrak L_0$ is a model of~$f^*L^e$
which is identified with~$L^{ed}$ via the fixed isomorphism~$\eps$.
Iterating this construction defines a sequence
$(\mathfrak X_n,\mathfrak L_n,ed^n)$ of models of~$(X,L)$,
with finite morphisms $f_n\colon\mathfrak X_n\ra \mathfrak X_{n-1}$
such that $f_n^*\mathfrak L_{n-1}=\mathfrak L_n$.
The metric on~$L$ defined by any of these models is semi-positive,
hence so is their uniform limit.
\end{proof}

\myparagraph{The canonical measure}
The measure $c_1(\bar L)^n$ on~$\mathrm X$
defined by the metrized line bundle~$\bar L$
is a very important invariant of the dynamical system.
It satisfies the functional equations
\[ f^* c_1(\bar L)^n = d^n c_1(\bar L)^n \quad\text{and}\quad
f_* c_1(\bar L)^n=c_1(\bar L)^n.
\]
The first follows by a general functorial property
proved in~\cite{chambert-loir2006};
it implies the second.
The support of the canonical measure is therefore totally
invariant under~$f$.

\myparagraph{The Fatou set}
Generalizing results of Kawaguchi--Silverman in~\cite{kawaguchi-silverman2009}
and Baker--Rumely~\cite{baker-rumely2010}, we want to show here that
the canonical measure
vanishes on any open set~$\mathrm U$ of~$\mathrm X$
where the sequence $(f^n|_{\mathrm U})$ of iterates of~$f$ is equicontinuous.

Let $\mathrm U$ be an open set in~$\mathrm X$
and $\mathscr F$ be a family of continuous maps from~$\mathrm U$
to~$\mathrm X$. One says that this family is equicontinuous
if for any $x\in\mathrm U$ and
any finite open covering $(\mathrm V_j)$ of~$\mathrm X$,
there exists a neighbourhood~$\mathrm U_x$
of~$x$ in~$\mathrm U$ such that for any $\phi\in\mathscr F$,
there exists an index~$j$ such that $\phi(\mathrm U_x)\subset \mathrm V_j$.
(This definition is adapted from Definition~10.63
in~\cite{baker-rumely2010};
it is the definition of equicontinuity associated to
the canonical uniform structure of the \emph{compact} space~$\mathrm X$.)

We define the equicontinuous locus of~$f$
as the largest open subset~$\mathrm E_f$ of~$\mathrm X$
over which the sequence of iterates of~$f$ is equicontinuous.


\begin{CLprop}
If $L$ is ample, then the metric~$\bar L$
is strongly pluriharmonic on~$\mathrm E_f$.\footnote{%
The ampleness assumption should not be necessary for the result to hold.}
\end{CLprop}
\begin{proof}
The proof is inspired from the above-mentioned sources,
which in turns is an adaptation of
the complex case~\cite{hubbard-papadopol1994}
(see also~\cite{ueda1994}).

We may replace~$L$ by a positive power of itself and
assume that it is very ample, induced by
a closed embedding of~$\mathrm X$ in~$\mathrm P^n$,
and that the natural map $\Gamma(\mathrm\P^n,\mathscr O(d))\ra\Gamma(\mathrm X,\mathscr O(d))$ is surjective.
Then, there are homogeneous polynomials $(F_0,\dots,F_n)$,
of degree~$d$, with coefficients in~$K$,
and without common zeroes on~$\mathrm X$,
such that $f([x_0:\dots:x_n])=[F_0(x):\dots:F_n(x)]$
for any $x=[x_0:\dots:x_n]\in\mathrm P^n$.
One considers the
polynomial map $F\colon\mathrm A^{n+1}\ra\mathrm A^{n+1}$;
it lifts a rational map on~$\mathrm P^n$
which extends the morphism~$f$.

For $(x_0,\dots,x_n)\in\mathrm A^{n+1}$, define
$\norm x=\max(\abs{x_0},\dots,\abs{x_n})$.
The Weil metric on~$\mathscr O(1)$ is given by
\[ \log\norm{s_P(x)}^{-1} =  \log\abs{P(x)}^{-1}
 + \deg(P) \log \norm{x}, \]
where $P$ is an homogeneous polynomial, $s_P$ the
corresponding global section of $\mathscr O(\deg(P))$, and
$x$ is a point of~$\mathrm A^{n+1}$ such that $P(x)\neq 0$.
The restriction to~$\mathrm X$
of this metric is a semi-positive metric~$\norm\cdot_0$ on~$L$.
The construction of the canonical metric on~$L$
introduces a sequence of semi-positive metrics~$\norm\cdot_n$
on~$L$; these metrics are given by the following explicit formula
\[ \log\norm{s_P(x)}_k^{-1}  = \log\abs{P(x)}^{-1}
 + \deg(P) d^{-k} \log\norm{F^{(k)}(x)}, \]
where $F^{(k)}:\mathrm A^{n+1}\ra\mathrm A^{n+1}$ is the $k$th
iterate of~$F$.

The convergence of this sequence is therefore equivalent to the
convergence of the sequence $(d^{-k}\log\norm{F^{(k)}})_k$
towards a continuous fonction on the preimage of~$\mathrm X$
under the projection map~$\mathrm A^{n+1}\setminus\{0\}\ra\mathrm P^n$.
The limit is usually called the homogeneous Green function.

For $0\leq i\leq n$,
let $\mathrm V_i$ be the open set of points~$x=[x_0:\dots:x_n]\in\mathrm P^n$
such that $\abs{x_i} >\frac12 \norm{x}$.
They form an open covering of~$\mathrm P^n$; their intersections
with~$\mathrm X$ form an open covering of~$\mathrm X$.

Fix $x\in \mathrm E_f$ and let $\mathrm U$ be an open neighbourhood
of~$x$ such that for any positive integer~$k$, there exists
$i\in\{0,\dots,n\}$ such that $f^k(\mathrm U)\subset \mathrm V_i$.
For any~$i$, let $N_i$ be the set of integers~$k$
such that $f^k(\mathrm U)\subset\mathrm V_i$.
Let us consider any index~$i$ such that~$N_i$ is infinite; to
fix ideas, let us assume that $i=0$.
The canonical norm of a section~$s_P$ at a point $y\in\mathrm U$
is given by
\begin{align*}
 \log\norm{s_P(y)}^{-1}
& =\log\abs{P(y)}^{-1} +\deg(P) \lim_{\substack{k\ra\infty \\k\in N_0}}
 d^{-k}   \log\norm{F^{(k)}(y)}  \\
& = \log\abs{P(y)}^{-1} \\
& \quad +\deg(P) \lim_{\substack{k\ra\infty \\k\in N_0}}
d^{-k} \bigg(  \log\abs{F_0^{(k)}(y)} \\
& \hskip 3cm {} +   \log \max_{0\leq i\leq m} \abs{F_i^{(k)}(y)/F_0^{(k)}(y)}\bigg).\end{align*}
Observe that $[F_0^{(k)}(y):\dots:F^{(k)}_m(y)]$ are the
homogeneous coordinates of the point~$f^k(y)$. Since $y\in\mathrm U$
and $f^k(\mathrm U)\subset\mathrm V_0$, one has
$\abs{F^{(k)}_i(y)}\leq 2\abs{F^{(k)}_0(y)}$, so that
the last term is bounded by $d^{-k}\log 2$ and uniformly
converges to~$0$ on~$\mathrm U$.
Finally, uniformly on~$\mathrm U$,
\[  \log\norm{s_P(y)}^{-1}
=  \log\abs{P(y)}^{-1} +\deg(P) \lim_{\substack{k\ra\infty \\k\in N_0}}
d^{-k}  \log\abs{F_0^{(k)}(y)}. \]
This shows that $\log\norm{s_P}^{-1}$ is strongly harmonic on~$\mathrm U$,
as claimed.
\end{proof}

\begin{CLcoro}
The canonical measure $c_1(\bar L)^n$ vanishes on~$\mathrm E_f$.
\end{CLcoro}
\begin{proof}
It suffices to apply Prop.~\ref{prop.local}.
\end{proof}

\myparagraph{Remarks}
1) The particular case~$\mathrm X=\mathrm P^n$
generalizes Theorem~6 in~\cite{kawaguchi-silverman2009}
according to which canonical metrics are locally constant
on the classical Fatou set (meaning that the norm of a non-vanishing
local section  is locally constant).
Indeed, the restriction to the set of smooth rigid points
of a strongly harmonic function is locally constant.
This follows from the fact that any such point has
an affinoid neighbourhood~$\mathrm U$ which is a polydisk,
so that the absolute value of any invertible function
on~$\mathrm U$, hence any harmonic  function on~$\mathrm U$
is constant.

2) In the case $\mathrm X=\mathrm P^1$, Fatou and Julia sets in
the Berkovich framework have been
studied by Rivera-Letelier~\cite{rivera-letelier2003}
and Benedetto~\cite{benedetto1998}; see also~\cite{baker-rumely2010}
for a detailed exposition of the theory and further references.
An example of Rivera-Letelier
on the projective line (Example~10.70 of~\cite{baker-rumely2010})
shows that the equicontinuity locus~$\mathrm E_f$ may be smaller
than the complement of the support of the measure $c_1(\bar L)$.

Anyway, this proposition suggests the interest
of a general study of Fatou sets and of pluripotential theory
on Berkovich spaces. For example, is there an interesting
theory of pseudoconvexity for Berkovich spaces? Is it related
to Stein spaces? By analogy to the complex case (see~\cite{ueda1994}),
are Berkovich Fatou components pseudoconvex? Stein?

\subsection{Abelian varieties}
Let us assume throughout this section that $\mathrm X$ is an Abelian variety.
For any integer~$m$, let $[m]$ be the multiplication-by-$m$
endomorphism of~$\mathrm X$.

\myparagraph{Canonical metrics}
Let $L$ be a line bundle on~$\mathrm X$. Let~$0$ be
the neutral element of~$\mathrm X$ and let us fix
a trivialization~$L_0$ of~$L$ at~$0$.

The line bundle~$L^{\otimes 2}$ is canonically decomposed
as the tensor product of an even and an odd line bundle:
\[ L^{\otimes 2} = (L\otimes [-1]^*L) \otimes (L\otimes [-1]^*L^{-1}). \]
By the theorem of the cube, an even line bundle~$L$ satisfies
$[m]^*L\simeq L^{\otimes m^2}$, while for an odd line bundle~$L$,
one has $[m]^*L \simeq L^{\otimes m}$; moreover, there
are in each case a unique isomorphism compatible
with the trivialization at the origin.
By Lemma~\ref{lemm.tate}, an even (resp. an odd) line bundle
possesses a canonical continuous metric making this
isomorphism an isometry.  This furnishes a \emph{canonical metric}
on~$L^{\otimes 2}$, hence on~$L$.
According to this lemma, this metric is semi-positive
if $L$ is ample and even.
Using a Lemma of Künnemann,
(\cite{chambert-loir2000b}, Lemme~2.3),
one proves that this also holds if $L$ is algebraically  equivalent to~$0$.
In any case, the canonical metrics are admissible.

\myparagraph{The case of good reduction}
When the variety~$\mathrm X$ has good reduction, the canonical
metrics and the associated measures are fairly easy to describe.
Indeed, let~$\mathfrak X$ be the Néron model  of~$\mathrm X$
over~$K^\circ$, an Abelian  scheme. For any line bundle~$L$ on~$\mathrm X$
there is a unique line bundle~$\mathfrak L$ on~$\mathfrak X$
which extends~$L$ and which admits a trivialization at the~$0$
section extending the given one over~$K$.
By the theorem of the cube for
the Abelian scheme~$\mathfrak X$, the isomorphism
$[m]^*L\simeq L^{\otimes m^a}$ (with $a=1$ or~$2$,
according to whether $L$ is odd or even)
extends uniquely to an isomorphism
$[m]^*\mathfrak L \simeq \mathfrak L^{\otimes m^a}$.
This implies that the canonical metrics are smooth,
induced by these models.

The description of the  canonical measures on~$\mathfrak X$ follows
at once.
Let $\xi$
be the  point of~$\mathrm X$ whose reduction
is the generic point of the special fiber of~$\mathfrak X$.
Then, for any family $(L_1,\dots,L_n)$ of line bundles on~$\mathrm X$,
one has
\[ c_1(\bar L_1)\dots c_1(\bar L_n) = \deg(c_1(L_1)\dots c_1(L_n)) \delta_\xi.
\]
We see in particular that they only depend on the classes
of the line bundles~$L_j$ modulo numerical equivalence.

\myparagraph{Gubler's description}
W.~Gubler~\cite{gubler2010} has computed the canonical measures
when the Abelian variety has bad reduction.
We describe his result here.

Up to replacing~$K$ by a finite extension, we assume
that $\mathrm X$ has split semi-stable reduction.
Raynaud's uniformization involves
an analytic group~$\mathrm E$
which is an extension of an abelian variety with
good reduction~$\mathrm Y$ by a split torus~$\mathrm T\simeq\gm^t$,
where $t\in\{1,\dots,n\}$ ---
the so-called Raynaud extension of~$\mathrm X$.
One has $t\geq 1$ since we assume \emph{bad} reduction;
moreover, $\dim \mathrm Y=\dim\mathrm E-t=n-t$.
There is a morphism $p\colon\mathrm E\ra\mathrm X$,
whose kernel is a discrete subgroup~$M$ of~$\mathrm E(K)$,
so that the induced map
$\mathrm E/\Lambda \ra \mathrm X$ is an isomorphism.
When $t=n$, one says that $\mathrm X$  has totally degenerate
reduction, and the morphism~$p$ is the rigid analytic uniformization
of the abelian variety~$\mathrm X$.

Moreover, $\mathrm E$ is constructed as
a contracted product $(\mathrm E_1\times \mathrm T)/\mathrm T_1$
from an extension~$\mathrm E_1$
of~$\mathrm Y$ by the ``unit subtorus'' $\mathrm T_1$
of~$\mathrm T$ (defined by the equalities $\abs{T_j(x)}=1$
for $j\in\{1,\dots,t\}$ and $x\in\mathrm T$).
The natural map $\lambda_{\mathrm T}\colon \mathrm T\ra\R^t$
defined by
\[ x\mapsto (-\log\abs{T_1(x)}, \dots, -\log\abs{T_t(x)}) \]
is continuous and surjective; it admits
a canonical section~$\iota_{\mathrm T}$
which maps a point $(u_1,\dots,u_t)\in\R^t$ to the semi-norm
\[ f\mapsto \sup_{\mathbf m\in\Z^t} a_{\mathbf m} e^{-m_1u_1-\dots-m_tu_t}
, \qquad \text{for $f=\sum_{\mathbf m} a_{\mathbf m} T_1^{m_1}\dots T_t^{m_t}\in\mathscr O(\mathrm T).$} \]

The map~$\lambda_{\mathrm T}$ extends uniquely to
a morphism~$\lambda\colon \mathrm E\ra \R^t$
whose kernel contains~$\mathrm E_1$.
The image~$\Lambda=\lambda(M)$ is a lattice of~$\R^t$,
and the morphism~$p$ induces a continuous proper morphism
$\rho\colon \mathrm X\ra \R^t/\Lambda$.
Composing the section~$\iota_{\mathrm T}$ with the projection~$p$
furnishes a section~$\iota\colon\R^t/\Lambda\ra\mathrm X$ of~$\rho$.
Its image is the \emph{skeleton} of~$\mathrm X$.
Gubler's theorem (\cite{gubler2010}, Cor.~7.3) is the following:
\begin{CLtheo}
Let $L_1,\dots,L_n$ be line bundles on~$\mathrm X$.
The canonical measure $c_1(\bar L_1)\dots c_1(\bar L_n)$
is the direct image by~$\iota$ of the unique Haar measure
on~$\R^t/\Lambda$ whose total mass is $\deg(L_1\dots L_n)$.
\end{CLtheo}

\section{Applications to Arakelov geometry}

We now describe some applications of the previous
considerations to arithmetic geometry over global fields.

\subsection{Adelic metrics and heights}

\myparagraph{Adelic metrics}
Let $F$ be either a number field (arithmetic case), or a finite extension of
the field of rational functions over a constant field (geometric case).
Let $M(F)$ be the set of normalized absolute values on~$F$.
Let $X$ be a projective variety over~$F$.
Any $v\in M(F)$ gives rise to a complete valued field~$F_v$,
and to an analytic space~$X_v$ over~$F_v$:
if $v$ is archimedean, $X_v=X(\overline{F_v})$,
while $X_v$ is the Berkovich analytic space attached to~$X_{F_v}$
if $v$ is ultrametric.

If $L$ is a line bundle on~$X$, an \emph{adelic metric}
on~$L$ is a family $(\norm\cdot_v)_{v\in M(F)}$
of continuous metrics on the induced line bundles over
the analytic spaces~$X_v$. We require the following supplementary compatibility
assumption: there exists a model~$(\mathfrak X,\mathscr L,e)$
over the ring of integers of~$F$ inducing the given
metrics at almost all places~$v$.
An adelic metric is said to be semi-positive, \resp
admissible if it is so at all places of~$F$.

Line bundles on~$X$ endowed with an adelic metric form a group
$\bPic(X)$; admissible line bundles form a subgroup $\bPicad(X)$.
If $f\colon Y\ra X$ is any morphism, there is a natural morphism
of groups $f^*\colon\bPic(X)\ra\bPic(Y)$; it maps
$\bPicad(X)$ into~$\bPicad(Y)$.

\myparagraph{Heights}
Consider line bundles $\bar L_0,\dots,\bar L_n$
with admissible adelic metrics. Let $Z$ be a subvariety
of~$X$ of dimension~$k$ and $s_0,\dots,s_k$
invertible meromorphic sections of~$L_0,\dots,L_k$
whose divisors have no common intersection point on~$Z$.
 For any $v\in M(F)$, we have recalled in Sections~\ref{sec.local.height.arch},
\ref{sec.local.height.arch.ad} and \ref{sec.local.height.berk.ad}
the definitions of the local height pairing
\[  (\hdiv(s_0)\dots \hdiv(s_k)|Z)_v \]
where the index~$v$ indicates the corresponding place of~$F$.
The global height is the sum, over all $v\in M(F)$, of these
local heights:
\[    (\hdiv(s_0)\dots \hdiv(s_k)|Z) =  \sum_{v\in M(F)}
  (\hdiv(s_0)\dots \hdiv(s_k)|Z)_v . \]
It inherits from the local heights their multilinear symmetric character.

Let us replace~$s_k$ by another invertible meromorphic section~$fs_k$.
Then,
\begin{align*} (\hdiv(s_0)\dots\hdiv(fs_k)|Z)
& = \sum_{v\in M(F)} (\hdiv(s_0)\dots\hdiv(fs_k)|Z)_v\\
& = \sum_{v\in M(F)} (\hdiv(s_0)\dots\hdiv(s_k)|Z)_v \\
& \quad
+ \sum_{v\in M(F)} \int_{X_v} \log\abs f^{-1} c_1(\bar L_0)\dots c_1(\bar L_{k-1}) \delta_{\mathrm Z_v}. \end{align*}
In particular, if $Z$ is a point $z\in X(F)$, then
$\delta_{\mathrm Z_v}=\delta_z$ is the Dirac mass at~$z$ and
\[ (\hdiv (s_0)|Z) = \sum_{v\in M(F)} \log\norm{s_0}^{-1}_v(z). \]
Let us observe that it is independent
on the choice of the chosen meromorphic section~$s_0$,
provided it is regular at~$z$.
Any other section has the form $fs_0$, for some invertible
meromorphic function~$f$ on~$X$. Then,
\begin{align*}
 (\hdiv (fs_0)|Z) & = \sum_{v\in M(F)} \log\norm{fs_0}^{-1}_v(z)  \\
& = \sum_{v\in M(F)} \log\norm{s_0}^{-1}_v(z)
+ \sum_{v\in M(F)} \log\abs f_v^{-1}(z) \\
& = (\hdiv (fs_0)|Z) \end{align*}
since, by the product formula, the second term vanishes.

By induction on the dimension of~$Z$, and using the commutativity
of the local height pairings, it follows that
the global height only depends on the metrized line bundles,
and not on the actual chosen sections~$s_0,\dots,s_k$.
We denote it by
\[ (\hc_1(\bar L_0)\dots \hc_1(\bar L_k)|Z). \]
Again, it is multilinear symmetric in the metrized line bundles~$\bar L_0,\dots, \bar L_k$. By the same argument, it only depends on their
isomorphism classes in $\bPicad(X)$.

It satisfies a projection formula: for any morphism $f\colon Y\ra X$
and any $k$-dimensional subvariety~$Z$ of~$Y$,
\[  (\hc_1(f^*\bar L_0)\dots \hc_1(f^*\bar L_k)|Z)
= (\hc_1(\bar L_0)\dots\hc_1(\bar L_k)| f_*(Z)), \]
where the cycle $f_*(Z)$ is defined as $\deg(Z/f(Z)) f(Z)$
if $Z$ and $f(Z)$ have the same dimension,
so that $f\colon Z\ra f(Z)$  is generically finite,
of some degree $\deg(Z/f(Z))$. If $Z$ and $f(Z)$
don't have the same dimension, one sets $f_*(Z)=0$.

\myparagraph{Heights of points}
The height of an algebraic point is an important tool
in Diophantine geometry.
If $\bar L$ is a line bundle
with an adelic metric on~$X$, then for any  point $P\in X(F)$,
viewed as a closed subscheme of~$X$, one has
\[ h_{\bar L} (P) = (\hc_1(\bar L)|P) = \sum_v \log\norm{s}^{-1}_v(P), \]
where $s$ is any meromorphic section on~$L$ which has neither
a zero nor a pole at~$P$. More generally, let $P\in X(\bar F)$
be an algebraic point and let $[P]$ be the corresponding closed point of~$X$.
Then,
\[ h_{\bar L}(P) = \frac1{[F(P):F]} (\hc_1(\bar L)|[P]) \]
is the height of~$P$ with respect to the metrized line bundle~$\bar L$.
In fact, restricted to points,
these definitions apply to any, not necessary admissible,
continuous metric on~$L$. The resulting function is a representative
of the classical height function relative to~$L$
(which is only defined up to the addition of a bounded function).

Observe also the following functorial property of the height: If $f\colon Y\ra X$ is a morphism and $P\in Y(\bar F)$, then $h_{f^*\bar L}(P)=h_{\bar L}(f(P))$.
Finally, recall that if $F$ is a global field,
then the height with respect to a metrized ample line bundle~$\bar L$
satisfies Northcott's finiteness property: for any integers~$d$
and~$B$, there are only finitely many points $P\in X(\bar F)$
such that $[F(P):F]\leq d$ and $h_{\bar L}(P)\leq B$.

\myparagraph{Zhang's inequality}
The essential minimum of the height~$h_{\bar L}$ is defined as
\[ e(\bar L)= \sup_{\substack{\emptyset\neq U\subset X}} \inf_{P\in U(\bar F)} h_{\bar L}(P), \]
where the supremum runs over non-empty open subsets of~$X$.
If $L$ is big, then $e(\bar L)$ is a real number. Another
way to state its definition is the following:
for any real number~$B$, then  the set
\[ \{ P\in X(\bar F)\,;\, h_{\bar L}(P)\leq B\} \]
is Zariski dense if $B>e(\bar L)$, and is not Zariski dense
if $B<e(\bar L)$.

Assume that $\bar L$  is an ample line bundle on~$X$,
equipped with a semi-positive adelic metric.
The (geometric/arithmetic) Hilbert-Samuel theorem implies the following
inequality
\[ e(\bar L) \geq \frac{(\hc_1(\bar L)^{n+1}|X)}{(n+1) (c_1(L)^n|X)}. \]
(See Zhang~\cite{zhang95b}, as well as~\cite{gubler2008,faber2009}
for more details in the geometric case). When $X$
is a curve and $F$ is a number field, Autissier~\cite{autissier2001b} proved
that the inequality holds for any ample line bundle
with an admissible adelic metric (see~\cite{chambert-loir2006});
this extends to the geometric case.

\subsection{Mahler measures and heights of divisors}

In this section, we assume that $X$ is a projective
geometrically integral smooth curve of positive genus~$g$ over~$F$.
For any place~$v\in M(F)$, let $X_v$ be the corresponding
analytic curve.

\myparagraph{}
Let $f$ be an invertible meromorphic function on~$X$. Let us view it
as an invertible meromorphic section of the trivial metrized line
bundle~$\overline{\mathscr O_X}$. Let $\bar L$ be any
line bundle on~$X$ with an admissible adelic metric.  Then,
\[  (\hc_1(\bar L)\hc_1(\overline{\mathscr O_X})|X) = 0. \]
Moreover, according to Theorem~1.3
of~\cite{chambert-loir-thuillier2009} (see Section~\ref{sec.mahler}),
\[  (\hc_1(\bar L)\hc_1(\overline{\mathscr O_X})|X)
=
 (\hc_1(\bar L)|\div(f)) + \sum_{v\in M(F)} \int_{X_v} \log\abs{f}_v^{-1}
  c_1(\bar L)_v. \]
In other words, this furnishes an integral formula
for the height (relative to~$\bar L$) of any divisor which
is rationally equivalent to~$0$:
\[ (\hc_1(\bar L)|\div(f)) = \sum_{v\in M(F)} \int_{X_v} \log\abs f_v c_1(\bar L)_v. \]

\myparagraph{Néron--Tate heights}
We want to apply this formula to a specific metrized line bundle on~$X$.
The Jacobian~$J$ of~$X$ is an Abelian variety of dimension~$g$.
We also choose a divisor~$D$ of degree~$1$ on~$X$ and
correspondingly fix an embedding~$\iota$ of~$X$ into~$J$.
(For this, we may need to enlarge the ground field~$F$.)
Finally, we let $\Theta$ be the theta divisor of~$J$,
defined as the image of~$X^{g-1}$ by the map
$(x_1,\dots,x_{g-1})\mapsto \sum_{j=1}^{g-1} \iota(x_j)$.

As described above, the line bundle~$\mathscr O_J(\Theta)$
admits a canonical metrization; this induces
a metrization on its inverse image~$L=\iota^* \mathscr O_J(\Theta)$ on~$X$.
The metrized line bundle~$\overline{\mathscr O_J(\Theta)}$ gives
rise to the (theta) Néron--Tate height on~$J$.
Consequently, decomposing $\div(f)=\sum n_P P$, we obtain
\begin{align*} (\hc_1(\bar L)|\div(f)) &= \sum n_P [F(P):F] \hat h_{\Theta}(\iota(P)) \\
& = \sum n_P [F(P):F] \hat h_{\Theta} ([P-D]),  \end{align*}
where $D$ is the fixed divisor of degree~$1$ on~$X$.

\myparagraph{Canonical measures}
Since $L$ has degree~$g$, the measure $c_1(\bar L)_v$
on~$X_v$ has total mass~$g$;
let us define a measure of total mass~$1$ on~$X_v$ by
\[  \mu_v=\frac1g c_1(\bar L)_v . \]

When $v$ is archimedean, the measure~$\mu_v$ is the Arakelov measure
on the Riemann surface $X_v(\C)$. Let us recall its definition.
Consider an orthonormal basis
$(\omega_1,\dots,\omega_g)$ of $H^0(X,\Omega^1_X)$, \ie,
a basis satisfying the relations
\[ \int_{X_v(\C)} \omega_j\wedge \overline{\omega_k} = \delta_{j,k}
= \begin{cases} 1 & \text{if $j=k$;} \\ 0 & \text{otherwise.}  \end{cases}\]
Then,
\[ \mu_v = \frac1g \sum_{j=1}^g \omega_j \wedge \overline{\omega_j}. \]

Let us now assume that $v$ is ultrametric.
By a theorem of Heinz~\cite{heinz2004},
the metric on the line bundle~$\bar L$  coincides
with the canonical metric defined by Zhang~\cite{zhang93}
using the reduction graph of the minimal regular model of~$X$;
see also~\cite{baker-faber2009} for a related interpretation
in the framework of tropical geometry.
This allows in particular to compute the measure~$\mu_v$:
the reader will find in~\cite[Lemma 3.7]{zhang93}
a quite explicit formula for~$\mu_v$,
involving the physical interpretation
of the graph as an electric network.
(Zhang's computation generalizes Theorem~2.11
of the prior paper~\cite{chinburg-r1993} by Chinburg and Rumely,
the normalization is slightly different; see also~\cite{baker-rumely2007}.)

\myparagraph{Superelliptic curves}
The formulas of this section combine to the following:
if $\div(f)=\sum n_P P$ is a divisor of an invertible meromorphic
function on~$X$,
\[ \sum n_P \hat h_{\Theta} ([ P-D]) = \sum_{v\in M(F)}
  \int_{X_v} \log\abs {f(x)}_v \, \mathrm d\mu_v(x). \]
As pointed out by R.~De~Jong~\cite{dejong2009},
the case of superelliptic curves is particularly interesting.
Indeed, such curves  are presented as a ramified $\mu_N$-covering
$x\colon X\ra\P^1$ of the projective line,
which is totally ramified over the point at infinity,
given by an equation $y^N=a(x)$, where $a$ is a polynomial
of degree~$m>N$, prime to~$N$. One has $g=\frac12 (N-1)(m-1)$.

Let us take
for the divisor~$D$
the single point~$O$ over the point at infinity.
For each point~$P$ in~$X(F)$, $x-x(P)$ is a rational function
on~$X$ which has a single pole of order~$N$ at infinity,
and which vanishes along the fiber~$x^{-1}(x(P))$ of~$x$.
The group of automorphisms of~$X$ acts transitively
on this fiber, and respects the metrics, so that
all of these points have the same Néron-Tate height.
This implies the following formula
\[  \hat h_{\Theta}(P-O)
     = \frac 1N \sum_{v\in M(F)} \int_{X_v} \log\abs{x-x(P)}_v\mu_v \]
of~\cite{dejong2009}. The elliptic Mahler measure,
defined
by~\cite{everest-fhlathuin1996,everest1999}
as a Shnirelman integral
is therefore
a natural integral when viewed on Berkovich spaces.

\subsection{An equidistribution theorem}

\myparagraph{Bogomolov's conjecture}
Let $X$ be a projective smooth curve  of genus~$g\geq 2$
and let $\bar L$ be an ample line bundle on~$X$
with a canonical metric inducing the Néron--Tate height.
When $F$ is a number field, Bogomolov conjectured
in~\cite{bogomolov80b}  that $e(\bar L)>0$; this
conjecture has been shown by Ullmo~\cite{ullmo98}.
Its generalization to a subvariety~$X$ of an Abelian variety~$A$,
$L$ being an ample line bundle on~$A$ with a canonical metric,
asserts  that $e(X,\bar L)>0$ when $X$ is not the translate
of an abelian subvariety by a torsion point; it has been shown
by Zhang~\cite{zhang98}.

Since $h_{\bar L}(P)=0$ for any algebraic point~$P\in A(\bar F)$
which is a torsion point, these theorems imply in turn
a theorem of Raynaud~\cite{raynaud83,raynaud83c}
(formerly, a conjecture of Manin and Mumford)
that the torsion points lying in a subvariety~$X$ of an abelian variety
are not Zariski dense in~$X$, unless $X$ is itself the translate
of an abelian subvariety by a torsion point.

The analogues of Bogomolov's and Zhang's conjecture in the geometric case
is still open in general; see~\cite{gubler2007b,cinkir2009}
and the references therein
for partial results.

\myparagraph{}
The proofs by Ullmo and Zhang of Bogomolov's conjecture
make a fundamental use of an equidistribution principle
which had been discovered together with Szpiro~\cite{szpiro-u-z97}.
Let us first introduce a terminology: say a sequence
(or a net)
of algebraic points in a variety~$X$ over a number
field is \emph{generic} if any strict subvariety of~$X$
contains at most finitely  many terms of the sequence.

Let $\bar L$ be a line bundle on~$X$ with a semi-positive
adelic metric.
The idea of the equidistribution principle is to consider
a generic sequence~$(x_j)$ such that $h_{\bar L}(x_j)\ra e(\bar L)$,
\ie, realizing the equality in Zhang's inequality,
and to use this inequality further, as a variational principle.
Let $v$ be a place of~$F$; for any~$n$, let $\delta(x_j)_v$
be the probability measure on~$X_v$ which gives any conjugate
of~$x_j$ the same mass, $1/[F(x_j):F]$. The equidistribution
theorem states that for a generic sequence~$(x_j)$,
the sequence of measures $(\delta(x_j)_v)$ on~$X_v$
converges vaguely towards the measure $c_1(\bar L)^n_v/(c_1(L)^n|X)$.

In these papers, the equidistribution property was only investigated
at an archi\-medean place, but the introduction of the measures
on Berkovich spaces was motivated by potential equidistribution theorems
on those.  In~\cite{chambert-loir2006}, I was able
to prove general results on curves only. Indeed, unless $X$
is a curve, I needed an ampleness assumption on the metrized line
bundle~$\bar L$ in order to apply Zhang's inequality
to slight variations of it.
This requirement has been removed by a paper of Yuan~\cite{yuan2008}
who could understand \emph{arithmetic volumes} beyond the ample case.
Yuan's proof is an arithmetic analogue of an inequality of Siu~\cite{siu1993}
which Faber~\cite{faber2009} and Gubler~\cite{gubler2008}
used to prove the geometric case of the equidistribution theorem.

In \cite{chambert-loir-thuillier2009}, we considered more general
variations of the metrized line bundles.
The discussion in that
article was restricted to the arithmetic case but the arguments
extend to the geometric case.

\begin{CLtheo}\label{theo.equidis}
Let $X$ be a projective variety of dimension~$n$ over~$F$.
Let $\bar L$ be an ample line bundle on~$X$ with a semi-positive
adelic metric such that $e(\bar L)=(\hc_1(\bar L)^{n+1}|X)=0$.
Let $(x_j)$ be  a generic sequence of algebraic points in~$X$
such that $h_{\bar L}(x_j)\ra 0$.
Then, for any  line bundle~$\bar M$ on~$X$ with
an admissible adelic metric,
\[  \lim_{j\ra\infty} h_{\bar M}(x_j) = \frac{(\hc_1(\bar L)^n\hc_1(\bar M)|X)}{(c_1(L)^n|X)}. \]
\end{CLtheo}
The particular case stated above is equivalent to
\loccit, Lemma~6.1,  as one can see by
by multiplying the metric on~$\bar L$ by an adequate constant at some place
of~$F$.
Taking for $M$ the trivial  line bundle~$\mathscr O_X$,
with an admissible metric,
one recovers the equidistribution theorems of Yuan, Faber
and Gubler.

\subsection{Lower bounds for heights and the Hodge index theorem}

In the final section, we use the Hodge
index theorem in Arakelov geometry to establish
positive lower bounds for heights on curves.
The results are inspired by recent papers~\cite{baker-demarco2009,
petsche-szpiro-tucker2009}, and the proofs are borrowed from~\cite{mimar1997}.
After they were conceived, I received the preprint~\cite{yuan-zhang2009}
which proves a similar result in any dimension.

\myparagraph{The arithmetic Hodge index theorem}
Let $X$ be a projective smooth curve over~$F$,
let $\bar L$ be a line bundle of degree~$0$ on~$X$, with
an admissible metric. Let $\bar L_0$ be the same line bundle
with the canonical metric: if $X$ has genus~$\geq 1$,
this is the metric induced by an embedding of~$X$ into its Jacobian,
if $X$ is of genus~$0$, then $\bar L_0$ is the trivial metrized
line bundle. The metrized line bundle~$\bar L\otimes\bar L_0^{-1}$
is the trivial line bundle, together with an admissible
metric which is given by a function~$f_v$ at the place~$v$ of~$F$.

A formula of Faltings--Hriljac
expresses~$(\hc_1(\bar L_0)^2|X)$ as minus twice the Néron--Tate height
of the point of~$J$ corresponding to~$L$. More generally,
\[ (\hc_1(\bar L)^2|X) = - 2 \hat h_{\mathrm{NT}}([L])
+ \sum_{v\in M(F)} \mathscr D(f_v), \]
where  for each~$v\in M(F)$,
\[ \mathscr D(f_v) = \int_{X_v} f_v \ddc (f_v) \]
is the \emph{Dirichlet energy} of~$f_v$. This is
a non-positive quadratic form which vanishes
if and only if $f_v$ is constant.
For more details, I refer to~\cite{bost1999} at archimedean places
and~\cite{thuillier2005} at ultrametric places.
(When $X$ has genus~$0$, $L\simeq \mathscr O_X$ and
the term $\hat h_{\mathrm{NT}}([L])$ has to be interpreted as~$0$.)

As a consequence, $(\hc_1(\bar L)^2|X)\leq 0$.
Let us analyse the case of equality. Since they are
nonpositive, all terms in the formula above have to vanish.
Consequently,  $[L]$ is a torsion point in the Jacobian,
and all functions~$f_v$ are constant. We will say that
some power of~$\bar L$ is constant

\begin{CLprop}\label{prop.hit}
Let $F$ be a number field, let $X$ be a projective smooth curve over~$F$.
Let $\bar L$ and~$\bar M$ be two admissible
metrized line bundles over~$X$.
Assume that $\deg(L)=\ell$, $\deg(M)=m$ are positive.
and $(\hc_1(\bar L)^2|X)=(\hc_1(\bar M)^2|X)=0$.
Then, the essential minimum of $\bar L\otimes \bar M$
satisfies the following inequality:
\[ e(\bar L\otimes \bar M) \geq - \frac1{2(\ell+m)\ell m}(\hc_1(m\bar L-\ell\bar M)^2|X). \]

Moreover, the  right hand side of this inequality is always nonnegative
and vanishes if and only if some power of
$\bar L^m\otimes\bar M^{-\ell}$ is constant.
\end{CLprop}
\begin{proof}
By Zhang's inequality (see~\cite{chambert-loir2006}),
one has
\[ e(\bar L+\bar M) \geq \frac1{2(\ell+m)} (\hc_1(\bar L+\bar M)^2|X). \]
Since $(\hc_1(\bar L)^2|X)=(\hc_1(\bar M)^2|X)=0$ by assumption,
we observe that
\[ (\hc_1(\bar L+\bar M)^2|X)=2(\hc_1(\bar L)\hc_1(\bar M)|X)
 = - \frac1{\ell m }(\hc_1(m\bar L-\ell\bar M)^2|X). \]
This shows the first claim.

Since $mL$ and~$\ell M$ have the same degree, viz.~$\ell m$,
the rest of the proposition
follows from the negativity properties of the height
recalled above.
\end{proof}

\myparagraph{}
Assume that $(x_n)$ is a generic sequence of points such that $h_{\bar L}(x_n)$
tends to~$0$.
By Theorem~\ref{theo.equidis},
$h_{\bar M}(x_n)$ converges to
$$ \frac1{\ell} (\hc_1(\bar L)\hc_1(\bar M)|X).$$
Except when both lower bounds are zero, this is strictly bigger
than the lower bound of the proposition, which is equal to
\[ \frac1{\ell+m} (\hc_1(\bar L)\hc_1(\bar M)|X). \]
In other words, the greedy obvious method to
find points of small height for~$\bar L+\bar M$
that first minimizes the height~$h_{\bar L}$,
only works up to the factor $(\ell+m)/\ell>1$.

\myparagraph{An example}
Let us give some explicit formulae
for the lower-bound above, in some particular cases.
We consider $X=\mathbf P^1$ over~$\Q$
and the metrized line bundle $\overline{\mathscr O(1)}_\W$.
Let $\phi$ and~$\psi$ be polynomials with integral coefficients,
 of degrees~$\ell$ and~$m$ respectively; let us put
$\bar L=\phi^*\overline{\mathscr O(1)}_\W$,
$\bar M=\psi^*\overline{\mathscr O(1)}_\W$. The line bundle
$\bar L^m\otimes\bar L^{-\ell}$ is trivial
and its metric is given by a family of functions~$(f_v)$.
Since $\phi$  and~$\psi$ have integral coefficients, $f_v=0$
at all finite places.
Moreover, since $g_{\bar L}(x)=\log\max(\abs{\phi(x)},1)$
and $g_{\bar M}(x)=\log\max(\abs{\psi(x)},1)$ are the Green
functions for the divisors~$\ell[\infty]$ and~$m[\infty]$ respectively,
one has
\[ f_\infty(x) = \log \frac{\max(\abs{\phi(x)}^m,1)}{\max(\abs{\psi(x)}^\ell,1)}. \]
Then,
\[ \ddc f_\infty = \frac m{2\pi}\mathrm d\Arg \phi(x) \wedge \delta_{\abs{\phi(x)}=1}
 - \frac \ell{2\pi}\mathrm d\Arg \psi(x) \wedge \delta_{\abs{\psi(x)}=1} \]
From this, we deduce that
\begin{align*}
 \mathscr D(f_\infty)
& =  \frac{\ell m}{2\pi}
\left(\int_{\abs{\psi(x)}=1} \log\max(\abs{\phi(x)},1)\mathrm d\Arg\psi(x)\right.\\
& \qquad {}\left.
   + \int_{\abs{\phi(x)}=1} \log\max(\abs{\psi(x)},1)\mathrm d\Arg\phi(x)\right),
\end{align*}
the two others terms vanishing.
In fact, Stokes's formula implies that the two terms within
the parentheses in  the previous
formula are equal and we have
\[ \mathscr D(f_\infty)=\frac{\ell m}{\pi} \int_{\abs{\phi(x)}=1}
 \log\max(\abs{\psi(x),1})\mathrm d\Arg\phi(x). \]

The simplest case to study is for $\phi(x)=x^\ell$. Then,
\[ \mathscr D(f_\infty)=\frac {\ell m}\pi \int_0^{2\pi} \log\max(\abs{\psi(e^{i\theta})},1)\,\mathrm d\theta  \]
is $2\ell m$ times the logarithm of the variant~$\mathrm M^+(\psi)$
of the Mahler measure of~$\psi$:
\[ \mathrm M^+(\psi)=  \exp\left( \frac1{2\pi} \int_0^{2\pi}
\log\max(\abs{\psi(e^{i\theta})},1)\,\mathrm d\theta \right). \]
In fact, Jensen's formula implies that
\[  \mathrm M^+(\psi)=  \exp\left( \frac1{(2\pi)^2} \int_0^{2\pi}
\log\abs{\psi(e^{i\theta_1})-e^{i\theta_2}}\, \mathrm d\theta_1\mathrm d\theta_2
\right) \]
is the Mahler measure $\mathrm M(\psi(x)-y)$
of the 2-variables polynomial $\psi(x)-y$.

Consequently, except for finitely many exceptions, any algebraic
point~$x\in \P^1(\bar\Q)$ satisfies
\[ \ell h(x) + h(\psi(x)) \geq \frac1{\ell +m} \log\mathrm M(\psi(x)-y). \]

For $\ell=1$ and $\psi(x)=1-x$, we obtain that up to finitely
many exceptions,
\[  h(x)+h(1-x)\geq \frac12 \log\mathrm M(1-x-y)\approx 0.161538,\]
In that particular case, Zagier~\cite{zagier1993} has proved a much
more precise result: except for 5 explicit points in~$\P^1$,
\[ h(x)+h(1-x)\geq \frac12 \log(\frac{1+\sqrt5}2)\approx 0.240606.\]
Observe also that if $(x_j)$ is a sequence of points such that $h(x_j)\ra 0$,
Theorem~\ref{theo.equidis} implies that
$h(1-x)\ra \log\mathrm M(1-x-y)\approx 0.323076$.

\myparagraph{Application to dynamical systems}
Let us assume that $\bar L$ and $\bar M$ are
the metrized line bundles~$\overline{\mathscr O(1)}_\phi$ and~$\overline{\mathscr O(1)}_\psi$
attached to rational functions~$\phi$ and~$\psi$ of degres~$d$
and~$e$ respectively, with $d\geq2$ and $e\geq 2$.
Let us write $h_\phi$ and $h_\psi$ for the height
relative to these metrized line bundles; we call them the
canonical heights.
The isometry $\phi^*\overline{\mathscr O(1)}_\phi\simeq
\overline{\mathscr O(1)}_\phi^d$ and the functorial properties
of the height imply that
for any $x\in \P^1(\bar F)$, $h_\phi(\phi(x))=d h_\phi(x)$
and $h_\psi(\psi(x))=e h_\psi(x)$. In particular, preperiodic
points for~$\phi$ (\ie, points with finite forward orbit)
satisfy $h_\phi(x)=0$. Moreover,
\begin{align*} d^2 (\hc_1(\overline{\mathscr O(1)}_\phi)^2|\P^1) &
=  (\hc_1(\phi^*\overline{\mathscr O(1)}_\phi)^2|\P^1) \\
& =  (\hc_1(\overline{\mathscr O(1)}_\phi)^2|\phi_*\P^1) \\
& =  d (\hc_1(\overline{\mathscr O(1)}_\phi)^2|\P^1),
 \end{align*}
hence $(\hc_1(\overline{\mathscr O(1)}_\phi)^2|\P^1)=0$
since $d\neq 0,1$.
Similarly, preperiodic points of~$\psi$ satisfy $h_\psi(x)=0$,
and
$(\hc_1(\overline{\mathscr O(1)}_\psi)^2|\P^1)=0$.

In the arithmetic case, or over function fields over a finite
field, Northcott's finiteness theorem implies easily
that points~$x$ such that $h_\phi(x)=0$ are preperiodic
for~$\phi$, and similarly for~$\psi$.
This is not true in general: for example, if $\phi$ is constant,
all constant points have height~$0$ but only countably many
of them are preperiodic; more generally isotrivial rational functions,
\ie rational functions which are constant after conjugacy by an automorphism
of~$\P^1$ will furnish counterexamples.
A theorem of Baker~\cite{baker2007} shows that the converse is true:
if $\phi$ is not isoconstant, then a point of height zero is then preperiodic;
the proof relies on a detailed analysis of a Green function
relative to the diagonal on~$\P^1\times\P^1$.
(In a more general context than the case of polarized dynamical
systems, Chatzidakis and Hrushovski~\cite{chatzidakis-hrushovski2008}
proved that the Zariski closure of the orbit
of a point of canonical height zero is isoconstant.)

Let us show how Prop.~\ref{prop.hit} implies some of the results of
Baker and DeMarco~\cite{baker-demarco2009}, and of
Petsche, Szpiro and Tucker~\cite{petsche-szpiro-tucker2009}.

\begin{CLprop}\label{prop.pst}
In the geometric case, let us assume that $\psi$ is
non-isotrivial;
if $F$ is a function field over an infinite field,
let us moreover assume that it is a polynomial.
The following are then equivalent:
\begin{enumerate}
\item the heights $h_\phi$ and $h_\psi$ coincide;
\item $\phi$  and $\psi$ have infinitely many common preperiodic points;
\item the essential lowest bound of~$h_\phi+h_\psi$ is zero;
\item the equilibrium measures~$\mu_\phi$ and~$\mu_\psi$ are equal
at all places;
\item the metrized line bundles $\mathscr O(1)_\phi$
and $\mathscr O(1)_\psi$ are isomorphic, up
to a family of constants $(c_v)$ such that $\prod c_v=1$.
\end{enumerate}
\end{CLprop}
\begin{proof}
The arguments are more or less formal from Prop.~\ref{prop.hit};
let us detail them anyway for the sake of the reader.

1)$\Rightarrow$2). Like any rational map, $\phi$ has infinitely many preperiodic
points in~$\P^1(\bar F)$, and they satisfy $h_\phi(x)=0$.
If $h_\phi=h_\psi$, then they also satisfy $h_\psi(x)=0$.
Under the assumptions of the proposition, they are preperiodic for~$\psi$.

2)$\Rightarrow$3) is obvious, for common preperiodic points
of~$\phi$ and~$\psi$ satisfy $h_\phi(x)+h_\psi(x)=0$.

3)$\Rightarrow$4). By Prop.~\ref{prop.hit}, the line bundle
$\mathscr O(1)_\phi-\mathscr O(1)_\psi$ has the constant metric at all places.
In particular, the local measures $\mu_\phi$
and $\mu_\psi$ coincide at all places.

4)$\Rightarrow$5). Let $s$ be a non zero global section of~$\mathscr O(1)$.
For any place~$v$, $f_v=\log(\norm s_{v,\phi}/\norm s_{v,\psi})$;
one has $\mu_{v,\psi}-\mu_{v,\phi}=\ddc f_v$, hence
$\ddc f_v=0$.
By the maximum principle of~\cite{thuillier2005},
$f_v$ is constant.
Moreover,
\[ 0=(\hc_1(\overline{\mathscr O(1)}_\psi)^2|X)=(\hc_1(\overline{\mathscr O(1)}_\phi)^2|X)
 + \sum_v \log c_v = \sum_v\log c_v. \]

5)$\Rightarrow$1). This is obvious.
\end{proof}

\myparagraph{Remarks}
1) The restrictive hypotheses on~$\psi$ have only been used to establish
the implication 1)$\Rightarrow$2).

2) Of course, many other results can be established by the same reasoning,
in particular the number field case of Theorem~1.2 of~\cite{baker-demarco2009}.
Let us also recall that the support of the equilibrium measure~$\mu_\phi$
is the Julia set~$J(\phi)$. If one can prove that
$J(\phi)\neq J(\psi)$ at some place, then
none of the assertions of Prop.~\ref{prop.pst} can possibly hold.
Similarly, Theorem~1.1 of that article can be seen as the conjunction
of our proposition and  an independent
complex analytic study of generalized Mandelbrot sets
(Prop.~3.3).

3) The main result of~\cite{yuan-zhang2009} is that a
variant of the implication (4)$\Rightarrow$(5) also holds in
a more general setting:
two semi-positive metrics on a line bundle which define
the same measure at a place~$v$ differ by multiplication
by a constant. The given proof works for curves.

4) We also recall that an implication similar to (1)$\Rightarrow$(5)
holds for general metrized line bundles
on arithmetic varieties, as proven by~\cite{agboola-pappas2000}:
if $\bar L$ and~$\bar M$ are line bundles with adelic metrics
such that $h_{\bar L}=h_{\bar M}$,
then $\bar L\otimes\bar M^{-1}$ is torsion in the
Arakelov Picard group~$\bPic(X)$: the heights determine
the metrics.


\bibliographystyle{smfplain}
\bibliography{aclab,acl,remarks}

\providecommand{\noopsort}[1]{}\providecommand{\url}[1]{\textit{#1}}
\providecommand{\bysame}{\leavevmode ---\ }
\providecommand{\og}{``}
\providecommand{\fg}{''}
\providecommand{\smfandname}{\&}
\providecommand{\smfedsname}{\'eds.}
\providecommand{\smfedname}{\'ed.}
\providecommand{\smfmastersthesisname}{M\'emoire}
\providecommand{\smfphdthesisname}{Th\`ese}
\begin{thebibliography}{64}
\expandafter\ifx\csname
natexlab\endcsname\relax\def\natexlab#1{#1}\fi
\expandafter\ifx\csname selectlanguage\endcsname\relax
  \def\selectlanguage#1{\relax}\fi

\bibitem[\protect\citename{Agboola and Pappas, }2000]{agboola-pappas2000}
Agboola, Adebisi, and Pappas, G. 2000.
\newblock Line bundles, rational points and ideal classes.
\newblock {\em Math. Res. Letters}, {\bf 7}, 709--717.

\bibitem[\protect\citename{Arakelov, }1974]{arakelov74}
Arakelov, Sergei~{\relax Ju}ri. 1974.
\newblock Intersection theory of divisors on an arithmetic surface.
\newblock {\em Izv. Akad. Nauk SSSR Ser. Mat.}, {\bf 38}(6), 1167--1180.

\bibitem[\protect\citename{Autissier, }2001]{autissier2001b}
Autissier, Pascal. 2001.
\newblock Points entiers sur les surfaces arithm\'etiques.
\newblock {\em J. reine angew. Math.}, {\bf 531}, 201--235.

\bibitem[\protect\citename{Baker, }2009]{baker2007}
Baker, Matthew. 2009.
\newblock A Finiteness Theorem for Canonical Heights Attached to Rational Maps
  over Function Fields.
\newblock {\em J. reine angew. Math.}, {\bf 626}, 205--233.
\newblock \url{arXiv:math.NT/0601046}.

\bibitem[\protect\citename{Baker and DeMarco, }2009]{baker-demarco2009}
Baker, Matthew, and DeMarco, Laura. 2009.
\newblock {\em Preperiodic points and unlikely intersections}.
\newblock arXiv:0911.0918.

\bibitem[\protect\citename{Baker and Faber, }2009]{baker-faber2009}
Baker, Matthew, and Faber, Xander. 2009.
\newblock Metric properties of the tropical {A}bel-{J}acobi map.
\newblock To appear in {\em Journal of Algebraic Combinatorics}.
\newblock arXiv:0905.1679.

\bibitem[\protect\citename{Baker and Rumely, }2007]{baker-rumely2007}
Baker, Matthew, and Rumely, Robert. 2007.
\newblock Harmonic Analysis on Metrized Graphs.
\newblock {\em Canad. J. Math.}, {\bf 59}(2), 225--275.

\bibitem[\protect\citename{Baker and Rumely, }2010]{baker-rumely2010}
Baker, Matthew, and Rumely, Robert. 2010.
\newblock {\em Potential Theory and Dynamics on the {B}erkovich Projective
  Line}.
\newblock Surveys and Mathematical Monographs, vol. 159.
\newblock Amer. Math. Soc.

\bibitem[\protect\citename{Bedford, }1993]{bedford93}
Bedford, Eric. 1993.
\newblock Survey of pluri-potential theory.
\newblock {In:} {\em Several complex variables}.
\newblock Math. Notes, no. ~38.
\newblock Stockholm, 1987/1988: Princeton Univ. Press.

\bibitem[\protect\citename{Bedford and Taylor, }1982]{bedford-t82}
Bedford, Eric, and Taylor, B.A. 1982.
\newblock A new capacity for plurisubharmonic functions.
\newblock {\em Acta Math.}, {\bf 149}, 1--40.

\bibitem[\protect\citename{Benedetto, }1998]{benedetto1998}
Benedetto, Robert~L. 1998.
\newblock Fatou components in {$p$}-adic dynamics.
\newblock {\em PhD Thesis.},  1--98.

\bibitem[\protect\citename{Berkovich, }1990]{berkovich1990}
Berkovich, Vladimir~G. 1990.
\newblock {\em Spectral theory and analytic geometry over non-{A}rchimedean
  fields}.
\newblock Mathematical Surveys and Monographs, vol. 33.
\newblock Providence, RI: American Mathematical Society.

\bibitem[\protect\citename{Berkovich, }1999]{berkovich1999}
Berkovich, Vladimir~G. 1999.
\newblock Smooth {$p$}-adic analytic spaces are locally contractible.
\newblock {\em Invent. Math.}, {\bf 137}(1), 1--84.

\bibitem[\protect\citename{Bloch {et~al.}, }1995]{bloch-g-s95b}
Bloch, Spencer, Gillet, Henri, and Soul{\'e}, Christophe. 1995.
\newblock Non-archimedean {A}rakelov theory.
\newblock {\em J. Algebraic Geometry}, {\bf 4}, 427--485.

\bibitem[\protect\citename{Bogomolov, }1980]{bogomolov80b}
Bogomolov, Fedor~Alekseivich. 1980.
\newblock Points of finte order on abelian varieties.
\newblock {\em Izv. Akad. Nauk. SSSR Ser. Mat.}, {\bf 44}(4), 782--804, 973.

\bibitem[\protect\citename{Bost, }1999]{bost1999}
Bost, Jean-Beno{\^\i}t. 1999.
\newblock Potential theory and {L}efschetz theorems for arithmetic surfaces.
\newblock {\em Ann. Sci. {\'E}cole Norm. Sup.}, {\bf 32}(2), 241--312.

\bibitem[\protect\citename{Call and Silverman, }1993]{call-s93}
Call, Gregory, and Silverman, Joseph. 1993.
\newblock Canonical heights on varieties with morphisms.
\newblock {\em Compositio Math.}, {\bf 89}, 163--205.

\bibitem[\protect\citename{Chambert-Loir, }2000]{chambert-loir2000b}
Chambert-Loir, Antoine. 2000.
\newblock Points de petite hauteur sur les vari{\'e}t{\'e}s
  semi-ab{\'e}liennes.
\newblock {\em Ann. Sci. {\'E}cole Norm. Sup.}, {\bf 33}(6), 789--821.

\bibitem[\protect\citename{Chambert-Loir, }2006]{chambert-loir2006}
Chambert-Loir, Antoine. 2006.
\newblock Mesures et équidistribution sur des espaces de {B}erkovich.
\newblock {\em J. reine angew. Math.}, {\bf 595}, 215--235.
\newblock math.NT/0304023.

\bibitem[\protect\citename{Chambert-Loir and Thuillier,
  }2009]{chambert-loir-thuillier2009}
Chambert-Loir, Antoine, and Thuillier, Amaury. 2009.
\newblock Mesures de {M}ahler et \'equidistribution logarithmique.
\newblock {\em Ann. Inst. Fourier (Grenoble)}, {\bf 59}(3), 977--1014.

\bibitem[\protect\citename{Chatzidakis and Hrushovski,
  }2008]{chatzidakis-hrushovski2008}
Chatzidakis, Zo{\'e}, and Hrushovski, Ehud. 2008.
\newblock Difference fields and descent in algebraic dynamics. {I}.
\newblock {\em J. Inst. Math. Jussieu}, {\bf 7}(4), 653--686.

\bibitem[\protect\citename{Chinburg and Rumely, }1993]{chinburg-r1993}
Chinburg, Ted, and Rumely, Robert. 1993.
\newblock The capacity pairing.
\newblock {\em J. reine angew. Math.}, {\bf 434}, 1--44.

\bibitem[\protect\citename{Cinkir, }2009]{cinkir2009}
Cinkir, Zubeyir. 2009.
\newblock {\em Zhang's Conjecture and the Effective {B}ogomolov Conjecture over
  function fields}.
\newblock arXiv:0901.3945.

\bibitem[\protect\citename{De~Jong, }2009]{dejong2009}
De~Jong, Robin. 2009.
\newblock {\em Canonical height and logarithmic equidistribution on
  superelliptic curves}.
\newblock arXiv:0911.1271.

\bibitem[\protect\citename{Deligne, }1987]{deligne1987}
Deligne, P. 1987.
\newblock Le d\'eterminant de la cohomologie.
\newblock {Pages  93--177 of:} {\em Current trends in arithmetical algebraic
  geometry ({A}rcata, {C}alif., 1985)}.
\newblock Contemp. Math., vol. 67.
\newblock Providence, RI: Amer. Math. Soc.

\bibitem[\protect\citename{Demailly, }1985]{demailly1985}
Demailly, Jean-Pierre. 1985.
\newblock Mesures de {M}onge-{A}mp\`ere et caract\'erisation g\'eom\'etrique
  des vari\'et\'es alg\'ebriques affines.
\newblock {\em M{\'e}m. Soc. Math. France},  124.

\bibitem[\protect\citename{Demailly, }1993]{demailly93}
Demailly, Jean-Pierre. 1993.
\newblock Monge-{A}mp{\`e}re operators, {L}elong numbers and Intersection
  theory.
\newblock {Pages  115--193 of:} {\em Complex analysis and geometry}.
\newblock Univ. Ser. Math.
\newblock New York: Plenum.

\bibitem[\protect\citename{Everest, }1999]{everest1999}
Everest, G.~R. 1999.
\newblock On the elliptic analogue of {J}ensen's formula.
\newblock {\em J. London Math. Soc. (2)}, {\bf 59}(1), 21--36.

\bibitem[\protect\citename{Everest and Fhlath{\'u}in,
  }1996]{everest-fhlathuin1996}
Everest, G.~R., and Fhlath{\'u}in, Br{\'{\i}}d~N{\'{\i}}. 1996.
\newblock The elliptic {M}ahler measure.
\newblock {\em Math. Proc. Cambridge Philos. Soc.}, {\bf 120}(1), 13--25.

\bibitem[\protect\citename{Faber, }2009]{faber2009}
Faber, X. W.~C. 2009.
\newblock Equidistribution of dynamically small subvarieties over the function
  field of a curve.
\newblock {\em Acta Arith.}, {\bf 137}(4), 345--389.

\bibitem[\protect\citename{Faltings, }1984]{faltings1984}
Faltings, Gerd. 1984.
\newblock Calculus on arithmetic surfaces.
\newblock {\em Ann. of Math. (2)}, {\bf 119}(2), 387--424.

\bibitem[\protect\citename{Favre and Jonsson, }2004]{favre-jonsson2004}
Favre, Charles, and Jonsson, Mattias. 2004.
\newblock {\em The valuative tree}.
\newblock Lecture Notes in Mathematics, vol. 1853.
\newblock Berlin: Springer-Verlag.

\bibitem[\protect\citename{Favre and Rivera-Letelier, }2007]{favre-rl2007}
Favre, Charles, and Rivera-Letelier, Juan. 2007.
\newblock {\em Th{\'e}orie ergodique des fractions rationnelles sur un corps
  ultram{\'e}trique}.
\newblock \url{arXiv:0709.0092}.

\bibitem[\protect\citename{Gillet and Soul{\'e}, }1990]{gillet-s90}
Gillet, Henri, and Soul{\'e}, Christophe. 1990.
\newblock Arithmetic intersection theory.
\newblock {\em Publ. Math. Inst. Hautes {\'E}tudes Sci.}, {\bf 72}, 94--174.

\bibitem[\protect\citename{Gubler, }1997]{gubler1997}
Gubler, Walter. 1997.
\newblock Heights of subvarieties over {$M$}-fields.
\newblock {Pages  190--227 of:} Catanese, F. (ed), {\em Arithmetic geometry}.
\newblock Symp. Math., vol. 37.

\bibitem[\protect\citename{Gubler, }1998]{gubler1998}
Gubler, Walter. 1998.
\newblock Local heights of subvarieties over non-archimedean fields.
\newblock {\em J. reine angew. Math.}, {\bf 498}, 61--113.

\bibitem[\protect\citename{Gubler, }2003]{gubler2003}
Gubler, Walter. 2003.
\newblock Local and canonical heights of subvarieties.
\newblock {\em Ann. Scuola Norm. Sup. Pisa}, {\bf 2}(4), 711--760.

\bibitem[\protect\citename{Gubler, }2007a]{gubler2007a}
Gubler, Walter. 2007a.
\newblock The {B}ogomolov conjecture for totally degenerate abelian varieties.
\newblock {\em Invent. Math.}, {\bf 169}(2), 377--400.

\bibitem[\protect\citename{Gubler, }2007b]{gubler2007b}
Gubler, Walter. 2007b.
\newblock Tropical varieties for non-{A}rchimedean analytic spaces.
\newblock {\em Invent. Math.}, {\bf 169}(2), 321--376.

\bibitem[\protect\citename{Gubler, }2008]{gubler2008}
Gubler, Walter. 2008.
\newblock Equidistribution over function fields.
\newblock {\em Manuscripta Math.}, {\bf 127}(4), 485--510.

\bibitem[\protect\citename{Gubler, }to appear]{gubler2010}
Gubler, Walter. to appear.
\newblock Non-archimedean canonical measures on abelian varieties.
\newblock {\em Compositio Math.}
\newblock arXiv:0801.4503.

\bibitem[\protect\citename{Heinz, }2004]{heinz2004}
Heinz, Niels. 2004.
\newblock Admissible metrics for line bundles on curves and abelian varieties
  over non-{A}rchimedean local fields.
\newblock {\em Arch. Math. (Basel)}, {\bf 82}(2), 128--139.

\bibitem[\protect\citename{Hubbard and Papadopol, }1994]{hubbard-papadopol1994}
Hubbard, John~H., and Papadopol, Peter. 1994.
\newblock Superattractive fixed points in {${\bf C}^n$}.
\newblock {\em Indiana Univ. Math. J.}, {\bf 43}(1), 321--365.

\bibitem[\protect\citename{Kawaguchi and Silverman,
  }2009]{kawaguchi-silverman2009}
Kawaguchi, Shu, and Silverman, Joseph~H. 2009.
\newblock Nonarchimedean {G}reen functions and dynamics on projective space.
\newblock {\em Math. Z.}, {\bf 262}(1), 173--197.

\bibitem[\protect\citename{Kontsevich and Soibelman,
  }2001]{kontsevich-soibelman2001}
Kontsevich, Maxim, and Soibelman, Yan. 2001.
\newblock Homological mirror symmetry and torus fibrations.
\newblock {Pages  203--263 of:} {\em Symplectic geometry and mirror symmetry
  ({S}eoul, 2000)}.
\newblock World Sci. Publ., River Edge, NJ.

\bibitem[\protect\citename{Kontsevich and Soibelman,
  }2006]{kontsevich-soibelman2006}
Kontsevich, Maxim, and Soibelman, Yan. 2006.
\newblock Affine structures and non-{A}rchimedean analytic spaces.
\newblock {Pages  321--385 of:} {\em The unity of mathematics}.
\newblock Progr. Math., vol. 244.
\newblock Boston, MA: Birkh\"auser Boston.

\bibitem[\protect\citename{Maillot, }2000]{maillot2000}
Maillot, Vincent. 2000.
\newblock G\'eom\'etrie d'{A}rakelov des vari\'et\'es toriques et fibr\'es en
  droites int\'egrables.
\newblock {\em M{\'e}m. Soc. Math. France},  129.

\bibitem[\protect\citename{Mimar, }1997]{mimar1997}
Mimar, Arman. 1997.
\newblock {\em On the preperiodic points of an endomorphism of {$\mathbf
  P^1\times\mathbf P^1$}}.
\newblock PhD Thesis, Columbia University.

\bibitem[\protect\citename{Petsche {et~al.}, }2009]{petsche-szpiro-tucker2009}
Petsche, Clayton, Szpiro, Lucien, and Tucker, Thomas~J. 2009.
\newblock {\em A dynamical pairing between two rational maps}.
\newblock arXiv:0911.1875.

\bibitem[\protect\citename{Raynaud, }1983a]{raynaud83}
Raynaud, Michel. 1983a.
\newblock Courbes sur une vari{\'e}t{\'e} ab{\'e}lienne et points de torsion.
\newblock {\em Invent. Math.}, {\bf 71}(1), 207--233.

\bibitem[\protect\citename{Raynaud, }1983b]{raynaud83c}
Raynaud, Michel. 1983b.
\newblock Sous-vari{\'e}t{\'e}s d'une vari{\'e}t{\'e} ab{\'e}lienne et points
  de torsion.
\newblock {Pages  327--352 of:} Artin, Michael, and Tate, John (eds), {\em
  Arithmetic and Geometry. Papers dedicated to I.R. Shafarevich}.
\newblock Progr. Math., no. ~35.
\newblock Birkh{\"a}user.

\bibitem[\protect\citename{Rivera-Letelier, }2003]{rivera-letelier2003}
Rivera-Letelier, Juan. 2003.
\newblock Dynamique des fonctions rationnelles sur des corps locaux.
\newblock {\em Ast\'erisque},  xv, 147--230.
\newblock Geometric methods in dynamics. II.

\bibitem[\protect\citename{Siu, }1993]{siu1993}
Siu, Yum~Tong. 1993.
\newblock An effective {M}atsusaka big theorem.
\newblock {\em Ann. Inst. Fourier (Grenoble)}, {\bf 43}(5), 1387--1405.

\bibitem[\protect\citename{Szpiro {et~al.}, }1997]{szpiro-u-z97}
Szpiro, Lucien, Ullmo, Emmanuel, and Zhang, Shou-Wu. 1997.
\newblock {\'E}quidistribution des petits points.
\newblock {\em Invent. Math.}, {\bf 127}, 337--348.

\bibitem[\protect\citename{Thuillier, }2005]{thuillier2005}
Thuillier, Amaury. 2005.
\newblock {\em Théorie du potentiel sur les courbes en géométrie non
  archimédienne. Applications à la théorie d'Arakelov}.
\newblock Ph.D. thesis, Université de Rennes 1.

\bibitem[\protect\citename{Ueda, }1994]{ueda1994}
Ueda, Tetsuo. 1994.
\newblock Fatou sets in complex dynamics on projective spaces.
\newblock {\em J. Math. Soc. Japan}, {\bf 46}(3), 545--555.

\bibitem[\protect\citename{Ullmo, }1998]{ullmo98}
Ullmo, Emmanuel. 1998.
\newblock Positivit{\'e} et discr{\'e}tion des points alg{\'e}briques des
  courbes.
\newblock {\em Ann. of Math.}, {\bf 147}(1), 167--179.

\bibitem[\protect\citename{Yuan and Zhang, }2009]{yuan-zhang2009}
Yuan, Xiniy, and Zhang, Shou-Wu. 2009.
\newblock {\em Calabi theorem and algebraic dynamics}.
\newblock
  \url{http://www.math.columbia.edu/~szhang/papers/calabi%20and%20semigroup.pd%
f}.

\bibitem[\protect\citename{Yuan, }2008]{yuan2008}
Yuan, Xinyi. 2008.
\newblock Big line bundles on arithmetic varieties.
\newblock {\em Invent. Math.}, {\bf 173}, 603--649.
\newblock \url{arXiv:math.NT/0612424}.

\bibitem[\protect\citename{Zagier, }1993]{zagier1993}
Zagier, Don. 1993.
\newblock Algebraic numbers close to both {$0$} and {$1$}.
\newblock {\em Math. Comp.}, {\bf 61}(203), 485--491.

\bibitem[\protect\citename{Zhang, }1993]{zhang93}
Zhang, Shou-Wu. 1993.
\newblock Admissible pairing on a curve.
\newblock {\em Invent. Math.}, {\bf 112}(1), 171--193.

\bibitem[\protect\citename{Zhang, }1995a]{zhang95}
Zhang, Shou-Wu. 1995a.
\newblock Positive line bundles on arithmetic varieties.
\newblock {\em J. Amer. Math. Soc.}, {\bf 8}, 187--221.

\bibitem[\protect\citename{Zhang, }1995b]{zhang95b}
Zhang, Shou-Wu. 1995b.
\newblock Small points and adelic metrics.
\newblock {\em J. Algebraic Geometry}, {\bf 4}, 281--300.

\bibitem[\protect\citename{Zhang, }1998]{zhang98}
Zhang, Shou-Wu. 1998.
\newblock Equidistribution of small points on abelian varieties.
\newblock {\em Ann. of Math.}, {\bf 147}(1), 159--165.

\end{thebibliography}

\end{document}